\documentclass[12pt]{amsart}

\usepackage{amssymb}
\usepackage{pb-diagram}
\usepackage{enumitem}
\usepackage[dvipsnames]{xcolor}
\usepackage{multicol}

\usepackage{graphicx}
\usepackage{amsmath}
\numberwithin{equation}{section}
\usepackage[all]{xy}

\newtheorem{Theorem}{Theorem}[section]
\newtheorem{Proposition}[Theorem]{Proposition}
\newtheorem{Lemma}[Theorem]{Lemma}
\newtheorem{Corollary}[Theorem]{Corollary}
\newtheorem{Definition}[Theorem]{Definition}
\newtheorem{Remark}[Theorem]{Remark}
\newtheorem{Example}[Theorem]{Example}

\numberwithin{equation}{section}

\begin{document}

\baselineskip=16pt

\title[A Note On the moduli space of rank two vector bundles on Hirzebruch surfaces]{A Note On the moduli space of rank two vector bundles on Hirzebruch surfaces}

\author[L. Roa-Leguizam\'on]{Leonardo Roa-Leguizam\'on}
\address{Universidade Estadual de Campinas (UNICAMP) \\ Instituto de Matemática, Estatísitica e Computação Científica (IMECC) \\ Departamento de Matem\'atica \\
Rua S\'ergio Buarque de Holanda, 651\ \ 13083-970 Campinas-SP, Brazil.} \email{leoroale@unicamp.br}

\thanks{}



\thanks{.}

\date{\today}

\begin{abstract} 
The aim of this paper is to study the Segre invariant for rank $2$ vector bundles on Hirzebruch surfaces   $\mathbb{F}_e$, $e \geq 0$.  We give necessary conditions in order to determine what numbers can appear as the Segre Invariant of a rank $2$ vector bundle on $\mathbb{F}_e$ with fixed Chern classes.  We compute a bound of the dimensions of the  strata whenever are non-empty. Finally, we present applications to  Brill-Noether Theory and describe differences between moduli spaces under change of polarization for rank $2$ vector bundles on $\mathbb{F}_e$.
\end{abstract}

\maketitle

\section{Introduction}

Let $X$ be a non-singular, irreducible complex projective surface and let $H$ be an ample divisor on $X$. Let $E$ be a vector bundle of rank $2$ on $X$ with Chern classes $c_1, c_2$. The Segre invariant is defined as
\[S_H(E) := \min\{c_1(E) \cdot H - 2 c_1(L) \cdot H\}
\]
where the minimum is taken over all line subbundles L of $E$.  This invariant induces a stratification of the moduli space $M_{X,H}(2,c_1,c_2)$ of $H-$stable vector bundles of rank $2$ on 
$X$ and fixed Chern classes $c_1$, $c_2$.  This stratification has been studied in \cite{ALH}  in order to get
topological and geometric properties of the moduli space $M_{\mathbb{P}^2}(2,c_1,c_2)$ and to give 
information about  Brill-Noether Theory
 for rank $2$ vector bundles on $\mathbb{P}^2$.  
 
The moduli space $M_{X,H}(n,c_1,c_2)$ was constructed by Maruyama \cite{Maruyama} in the 1970's and it has been studied by several authors.  However,  little is known about its geometry and
subvarieties.  In the case, $X= \mathbb{F}_e$  and $H$ an ample divisor on $\mathbb{F}_e$, there are general results about of moduli space $M_{\mathbb{F}_e,H}(2,c_1,c_2)$ (see for instance, \cite{Costa-Miro-Roig}, \cite{Aprodu} and \cite{Aprodu-Brinzanescu}). In particular, $M_{\mathbb{F}_e,H}(2,c_1,c_2)$ is smooth, irreducible, rational, quasi-projective variety of dimension $4c_2-c_1^2-3$, whenever is non-empty.

The aim of this paper is to use the Segre
invariant for rank $2$ vector bundles on Hirzebruch surfaces $\mathbb{F}_e$, $e \geq 0$, to give information about the moduli space $M_{\mathbb{F}_e,H}(2;-\alpha C_e-\beta F,c_2)$, $\alpha ,\beta \in \{0,1\}$ where $H:= C_e+(e+1)F$ is an ample divisor on $\mathbb{F}_e$.

Section 2 introduces the Segre invariant on surfaces,  and collects
a number of results on Hirzebruch surfaces that will be subsequently used. Section 3 is
the core of the paper. The main issue is to give necessary conditions to determine 
what numbers
can appear as the Segre invariant of a vector bundle on
$\mathbb{F}_e$ with fixed characteristic classes. The answer is
given by the following results:

\textbf{Theorem \ref{TP1}}
Let $H=C_e+(e+1)F$   be an ample divisor on $\mathbb{F}_e$, $e \geq 0$, and let \\
$\alpha, \beta \in \{0,1\}$, $c_2 \geq 2$, $a,b \in \mathbb{N}$, $a  \neq 0$ and set $s:=2(a-b)-\alpha-\beta >0$.  Let $Z \subset \mathbb{F}_e$ be a local complete intersection of codimension two of lenght $l(Z)=c_2-a^2e+ae\alpha-2ab -a\beta+b\alpha > 0$.  If \[
    \begin{aligned}  
        h^0(\mathbb{F}_e,  \mathcal{O}_{\mathbb{F}_e}((a-\alpha)C_e+ & (2(a-b)-\alpha-\beta-1)F \otimes I_Z)=  \\ & h^0(\mathbb{F}_e, \mathcal{O}_{\mathbb{F}_e}((2a-\alpha-1)C_e+(a-2b-\beta-1)F \otimes I_Z)=0, 
    \end{aligned}
\]
then there exists an $H-$stable vector bundle $E$  with Chern classes $c_1(E) = -\alpha C_e-\beta F$, $c_2$ and $S_H(E)=s$. 
Furthermore, $E$ can be written in the following exact sequence:
\[0 \rightarrow \mathcal{O}_{\mathbb{F}_e}(-aC_e+bF) \rightarrow E \rightarrow \mathcal{O}_{\mathbb{F}_e}((a-\alpha)C_e-(b+\beta)F) \otimes I_Z \rightarrow 0,\]
with  $\mathcal{O}_{\mathbb{F}_e}(-aC_e+bF) \subset E$ maximal.

\textbf{Theorem \ref{TP2}}
Let $H=C_e+(e+1)F$   be an ample divisor on $\mathbb{F}_e$, $e \geq 0$, and let $\alpha, \beta \in \{0,1\}$, $c_2 \geq 2$, $a,b \in \mathbb{N}$, $b  \neq 0$ and set $s:=2(b-a)-\alpha-\beta >0$.  Let $Z \subset \mathbb{F}_e$ be a local complete intersection of codimension two of lenght $l(Z)=c_2-a^2e-ae\alpha-2ab +a\beta-b\alpha > 0$.  If   
    \[ 
    \begin{aligned}
        h^0(\mathbb{F}_e,  \mathcal{O}_{\mathbb{F}_e}& (2(b-a)-\alpha-\beta-1)C_e+  (b-\beta)F \otimes I_Z)=  \\ & h^0(\mathbb{F}_e, \mathcal{O}_{\mathbb{F}_e}((b-2a-\alpha-1)C_e+(2b-\beta-1)F \otimes I_Z)=0, 
    \end{aligned}
    \]
then there exists an $H-$stable vector bundle $E$  with Chern classes $c_1(E) = -\alpha C_e-\beta F$, $c_2$ and $S_H(E)=s$. 
Furthermore, $E$ can be written in the following exact sequence:
\[0 \rightarrow \mathcal{O}_{\mathbb{F}_e}(aC_e-bF) \rightarrow E \rightarrow \mathcal{O}_{\mathbb{F}_e}(-(a+\alpha)C_e+(b-\beta)F) \otimes I_Z \rightarrow 0\]
with  $\mathcal{O}_{\mathbb{F}_e}(aC_e-bF) \subset E$ maximal. 

\textbf{Theorem \ref{TP3}}
 Let $H=C_e+(e+1)F$  be an ample divisor on $\mathbb{F}_e$, $e \geq 0$, and let $\alpha, \beta \in \{0,1\}$, $c_2 \geq 2$, $a,b \in \mathbb{N}-\{0\}$ and set $s:=2(a+b)-\alpha-\beta >0$.  Let $Z \subset \mathbb{F}_e$ be a local complete intersection of codimension two of lenght $l(Z)=c_2-a^2e+ae\alpha+2ab -a\beta-b\alpha >0$ satisfying any of the conditions of item (3), Lemma \ref{nosecciones}.   If 
 the pair 
\[
(\mathcal{O}_{\mathbb{F}_e}((2a-2)C_e + (2b-e-2-\beta)F), I_Z)
\]
satifies the property of Cayley-Bacharach (See Theorem \ref{CB}) and 
    \[
    \begin{aligned}
        h^0(\mathbb{F}_e,  \mathcal{O}_{\mathbb{F}_e}&  (2(a+b)-\alpha-\beta-1)C_e+  (b-\beta)F \otimes I_Z)= \\ &  h^0(\mathbb{F}_e, \mathcal{O}_{\mathbb{F}_e}((a-\alpha)C_e+  (2(a+b)-\alpha-\beta-1)F \otimes I_Z)= \\ & h^0(\mathbb{F}_e, \mathcal{O}_{\mathbb{F}_e}((2a-\alpha-1)C_e+  (a+2b-\beta-1)F \otimes I_Z) = \\ &
        h^0(\mathbb{F}_e, \mathcal{O}_{\mathbb{F}_e}((2a+b-\alpha-1)C_e+  (2b-\beta-1)F \otimes I_Z)=0, 
    \end{aligned}
    \]
then there exists an $H-$stable vector bundle $E$  with Chern classes $c_1(E) = -\alpha C_e-\beta F$, $c_2$ and $S_H(E)=s$. 
Furthermore, $E$ can be written in the following exact sequence:
\[0 \rightarrow \mathcal{O}_{\mathbb{F}_e}(-aC_e-bF) \rightarrow E \rightarrow \mathcal{O}_{\mathbb{F}_e}((a-\alpha)C_e+(b-\beta)F) \otimes I_Z \rightarrow 0\]
with  $\mathcal{O}_{\mathbb{F}_e}(-aC_e-bF) \subset E$ maximal.

In section 4. we prove  that if the stratum 
$M_{\mathbb{F}_e}(2;-\alpha C_e-\beta F,c_2;s)$ is non-empty, then the elements are parameterized by extensions. Thus, a bound of the dimension of the stratum is
obtained by \lq\lq counting parameters" of such extensions. We have the following results: 

\textbf{Theorem \ref{dimmax1}} 
Let $H=C_e+(e+1)F$  be an ample divisor on $\mathbb{F}_e$, $e \geq 0$ and let $ \beta \in \{0,1\}$, $c_2 \geq 2$, $s, l_1,l_2, l_3 > 0$.  Suppose that there exist $a_i,b_i \in \mathbb{N}$,  for $i=1,2,3$ such that 
\[
\begin{aligned}
    s  = 2(a_1-b_1)-\beta & = 2(b_2-a_2)-\beta = 2(a_3+b_3)-\beta, \,\,\, a_1, b_2 ,a_3,b_3 \neq 0 \\
    l_1 &= c_2-a_1^2e-2a_1b_1-a_1\beta \\
    l_2 &= c_2-a_2^2e-2a_2b_2 +a_2\beta \\
    l_3 &= c_2-a_3^2e+2a_3b_3-a_3\beta
\end{aligned}\] 
Let $D_1:= a_1C_e-b_1F$ be a divisor on $\mathbb{F}_e$, and let $Z_1 \subset \mathbb{F}_e$ be a local complete intersection of codimension two of lenght $l_1$,   satisfiying the hypothesis of  Theorem \ref{TP1} taking $\alpha =0$. Let $D_2:= -a_2C_e+b_2F$ be a divisor on $\mathbb{F}_e$, and  let $Z_2 \subset \mathbb{F}_e$ be a local complete intersection of codimension two of lenght $l_2$,   satisfiying the hypothesis of  Theorem \ref{TP2} taking $\alpha =0$.  Let $D_3:= a_3C_e+b_3F$ be a divisor on $\mathbb{F}_e$,  and let $Z_3 \subset \mathbb{F}_e$ be a local complete intersection of codimension two of lenght $l_3$,   satisfiying the hypothesis of  Theorem \ref{TP3} taking $\alpha =0$. Then, the stratum $M_{\mathbb{F}_e,H}(2;-\beta F,c_2;s)$ is non-empty and it has an irredubible subvariety of dimension
\[\begin{aligned}
     \max_{a_i,b_i} \{ 3l+ \frac{1}{2} (2a_1-1)(4b_1+2a_1e+2\beta+2),  & \,\,  3l+ \frac{1}{2}(2a_2+1)(4b_2-2a_2e-2\beta-4e-2), \\ & 3l+ \frac{1}{2}(2a_3-1)(2\beta+2a_3e-4b_3+2) \}
\end{aligned}\]
where the maximum is taken over all numbers $a_i,b_i$ that satisfies  the hypothesis of theorem. Moreover, 
\[\begin{aligned}
   \dim\, & M_{\mathbb{F}_e}(2; -\beta F,c_2;s) \geq 
    \max_{a_i,b_i} \{ 3l+ \frac{1}{2} (2a_1-1)(4b_1+2a_1e+2\beta+2), \\ & 3l+ \frac{1}{2}(2a_2+1)(4b_2-2a_2e-2\beta-4e-2),  \,\, 3l+ \frac{1}{2}(2a_3-1)(2\beta+2a_3e-4b_3+2) \}
\end{aligned}\]    

\textbf{Theorem \ref{dimmax2}} 
Let $H=C_e+(e+1)F$  be an ample divisor on $\mathbb{F}_e$, $e \geq 0$ and let $ \beta \in \{0,1\}$, $c_2 \geq 2$, $s > 0$.  Suppose that there exist $a_i,b_i \in \mathbb{N}$, $b_i \neq 0$ for $i=1,2,3$ such that  \[\begin{aligned}
    s= 2(a_1-b_1)-\beta-1 & = 2(b_2-a_2)-\beta-1 = 2(a_3+b_3)-\beta-1 \,\,\, a_2, b_1 ,a_3,b_3 \neq 0 \\
    l_1 & = c_2-a_1^2e+a_1e-2a_1b_1+b_1-a_1\beta\\
    l_2 & =  c_2-a_2^2e-a_2e-2a_2b_2-b_2 +a_2\beta \\
    l_3 & = c_2-a_3^2e+a_3e+2a_3b_3-b_3 -a_3\beta\\
\end{aligned}\] 
Let $D_1:= a_1C_e-b_1F$ be a divisor on $\mathbb{F}_e$, and  let $Z_1 \subset \mathbb{F}_e$ be a local complete intersection of codimension two and lenght $l_1$,   satisfiying the hypothesis of  Theorem \ref{TP1} taking $\alpha =1$. Let $D_2:= -a_2C_e+b_2F$ be a divisor on $\mathbb{F}_e$, and  let $Z_2 \subset \mathbb{F}_e$ be a local complete intersection of codimension two and lenght $l_2$,   satisfiying the hypothesis of  Theorem \ref{TP2} taking $\alpha =1$. Let $D_3:= a_3C_e+b_3F$ be a divisor on $\mathbb{F}_e$,   and let $Z_3 \subset \mathbb{F}_e$ be a local complete intersection of codimension two and lenght $l_3 $,   satisfiying the hypothesis of  Theorem \ref{TP3} taking $\alpha =1$. Then, the stratum $M_{\mathbb{F}_e,H}(2;- C_e-\beta F,c_2;s)$ is non-empty and it has an irredubible subvariety of dimension
\[\begin{aligned}
\max_{a_i,b_i}\{
   & 3l+ \frac{1}{2} (2a_1-2)(4b_1+2a_1e- e+2\beta+2), \\
   & 3l+ \frac{1}{2}(2a_2+2)(4b_2-2a_2e- e-2\beta-4e-2), \,\,  3l+ \frac{1}{2}(2a_3-2)(2\beta+2a_3e -4b_3+1) \}
   \end{aligned}
\]
where the maximum is taken over all numbers $a_i,b_i$ that satisfies  the hypothesis of theorem. Moreover, 
\[\begin{aligned}
   \dim\, & M_{\mathbb{F}_e}(2;- C_e -\beta F,c_2;s) \geq 
    \max_{a_i,b_i}\{
   3l+ \frac{1}{2} (2a_1-2)(4b_1+2a_1e- e+2\beta+2), \\
   & 3l+ \frac{1}{2}(2a_2+2)(4b_2-2a_2e- e-2\beta-4e-2), \,\, 3l+ \frac{1}{2}(2a_3-2)(2\beta+2a_3e -4b_3+1) \}
   \end{aligned}
\]

We finish the paper by computing a lower bound for the dimension
of the Brill-Noether varieties. Thus, we can show the existence of nonempty Brill-Noether locus with
negative Brill-Noether.  Also, we give a  description of the differences between the moduli spaces $M_{\mathbb{F}_e, H_1}(2;-C_e-\beta F,c_2)$ and $M_{\mathbb{F}_e, H_t}(2;-C_e-\beta F,c_2)$, $t > 0$ where $\beta \in \{0,1\}$ and  $H_t:= C_e+(e+t)F$, $t> 0$ is an ample divisor on $\mathbb{F}_e$.


\section{Preliminaries}

We start this section by recalling the main results about Segre invariant for rank $2$ vector bundles on surfaces, for more details see \cite{ALH},  and we recall the main results that we will use on Hirzebruch surfaces $\mathbb{F}_e$, $e \geq 0$, and on cohomology of line bundles on $\mathbb{F}_e$. For a further treatment of the subject see \cite{Hartshorne1} and \cite{Coskun-Huizenga}.

\subsection{Segre invariant on surfaces}

 Let $X$ be a smooth, irreducible complex projective surface and let $H$ be an ample divisor on $X$. For  a vector bundle $E$ of rank $2$ with Chern classes $c_1, c_2$ on $X$  the \textit{Segre invariant} $S_H(E)$ is defined as 
\[S_H(E) = c_1(E) \cdot H - 2 \max \{c_1(L) \cdot H\}\]
where the maximum is taken over all subline bundles $L$ of $E$.  Recall that the $H-$slope of a rank $2$ vector bundle $E$ denoted by $\mu_H(E)$, is the quotient
\[\mu_H(E) := \frac{c_1(E) \cdot H}{2}\]
and $E$ is called $H-$stable, if for all line bundle $L \subset E$, we have
\[c_1(L) \cdot H < \mu_H(E)\]
So, the Segre invariant can be wirtten as
\[S_H(E) = 2 \min_{L \subset E} \{\mu_H(E) - c_1(L) \cdot H\}\]
where the minimum is taken over all line subbundles $L$ of $E$.

The Segre invariant is always a finite number.

\begin{Lemma} Let  $H$ be  an ample line bundle on $X$ and let $E$ be a rank $2$ vector bundle on $X$. Then,
the set:
$$\{ L \cdot H : L\subset E, L \text{ a line bundle} \},$$
is bounded from above.
\end{Lemma}

\begin{proof} 
The result follows as 
a slight generalization of \cite[Lemma 1.1]{Maruyama2} for which the same proof works.
\end{proof}

The term invariant is used because $S_H(E) = S_H(E \otimes L)$ for any line bundle $L\in Pic(X)$ and $E$ is $H-$ stable if and only if $S_H(E) > 0$. We say that $L \subset E$ is $H-$\textit{maximal} if
\[S_H(E) = c_1(E) \cdot H - 2c_1(L) \cdot H.\]
For simplicity of notation, for a fixed $H$ we write maximal instead $H-$maximal. 

\begin{Remark}\label{NotUnique}\emph{
Let $H$  be an ample divisor on $X$ and let $E$ be a rank $2$ vector bundle on $X$ with Chern classes $c_1$, $c_2$.  If  $S_H(E):=s$, then there exists a line bundle $L \subset E$ such that
\[s = c_1(E) \cdot H -2(c_1(L) \cdot H)\]
which is equivalent to 
\[c_1(L) \cdot H = \frac{c_1(E) \cdot H-s}{2}\]
and $E$ can be written in the exact sequence
\[0 \rightarrow L \rightarrow E \rightarrow L'\otimes I_Z \rightarrow 0\]
where $Z$ is a locally complete interesection of codimension two of length $l(Z) = c_2-(c_1(L)\cdot H)(c_1(L')\cdot H)$.  When $X$ is a Hirzebruch surface $\mathbb{F}_e$, $e \geq 0$
the line bundle $L$ is not unique, because  $Pic(\mathbb{F}_e)= \mathbb{Z}C_e \oplus \mathbb{Z} F$, where $C_e$ is the class of the unique section of self-intersection $-e$ and $F$ is the class of a fiber of the projection to $\mathbb{P}^1$.}
\end{Remark}

\begin{Remark}\emph{
In general, a subline bundle does not define a torsion-free quotient.  Howe-\\ver, any maximal line bundle $L \subset E$ defines a corresponding torsion free quotient which fit in the exact sequence
\[0 \rightarrow L \rightarrow E \rightarrow L' \otimes I_Z \rightarrow 0,\]
where $Z \subset X$ is a local complete intersection of codimension $2$, (see \cite[Proposition 5.]{Friedman}).}
\end{Remark}

Segre invariant induces a stratification of the moduli space of vector bundles:

\begin{Lemma}\label{semicontinuous}
Let  $H$ be an ample divisor on $X$ and let $T$ be a variety. Let $\mathcal{E}$ be a vector bundle of rank $2$ on $X \times T$.  The function
\begin{eqnarray*}
S_H: T &\longrightarrow& \mathbb{Z} \\
t &\longmapsto& S_H(\mathcal{E}_t)
\end{eqnarray*}
is lower semicontinuous.
\end{Lemma}

\begin{proof}
The semicontinuity follows as a slight generalization of the openness property of stability for which the same proof works (see \cite[Theorem 2.8]{Maruyama})
\end{proof}

The moduli space of $H$-stable vector bundles with fixed Chern classes $c_1$ and $c_2$ on $X$ were constructed in the 1970's by
Maruyama (see \cite{Maruyama1}).  We shall denote the moduli space of $H-$stable vector bundles of rank $n$ and Chern classes $c_1$ and $c_2$ on $X$ by $M_{X,H}(n; c_1,c_2)$. In case $X$ is a smooth, projective rational surface the moduli space $\dim M_{X,H}(2, c_1,c_2) $ is either empty or a smooth irreducible variety of dimension   $\dim M_{X,H}(2, c_1,c_2)=4c_2-c_1^2-3$.

\subsection{Hirzebruch surfaces}

Let $e \geq 0$ be a non-negative integer and let $\mathbb{F}_e$ denote the ruled surface $\mathbb{P}(\mathcal{O}_{\mathbb{P}^1} \oplus \mathcal{O}_{\mathbb{P}^1}(e))$. Let $\pi: \mathbb{F}_e \rightarrow \mathbb{P}^1$ be the natural projection.  When $e \geq 1$, let $C_e$ be the class of the unique section of self-intersection $-e$ and let $F$ denote the class of a fiber of the projection to $\mathbb{P}^1$. The surface $\mathbb{F}_0$ is isomorphic to $\mathbb{P}^1 \times \mathbb{P}^1$. In that case, let $C_e$ and $F$ denote the classes of the two rulings.  The intersection pairing on $\mathbb{F}_e$ is given by
\[C_e^2=-e, \,\,\,\, F^2= 0, \,\,\,\, C_e \cdot F =1. \]
A  canonical divisor $K_{\mathbb{F}_e}$ is given by $ -2C_e-(e+2)F$. $Pic(\mathbb{F}_e) \cong \mathbb{Z} C_e \oplus \mathbb{Z}F$ and $\chi(\mathcal{O}_{\mathbb{F}_e})=1$.

From Riemann- Roch Theorem, it follows that 
\[\chi(\mathcal{O}_{\mathbb{F}_e}(aC_e+bF))= (a+1)(b+1) -e \frac{a(a+1)}{2}.\]

\begin{Remark} \label{efample} \emph{(\cite[Corollary 2.18]{Hartshorne1})
Let $D$ be the divisor $aC_e+bF$ on the Hirzebruch surface $\mathbb{F}_e$, $e \geq 0$.  Then
\begin{itemize}
\item [(i)] $D$ is effective if and only if $a, b \geq 0$.
\item [(ii)] $D$ is ample if and only if $a > 0$ and $b > ae$.
\end{itemize}}
By Serre duality, it terms out that  $h^2(\mathbb{F}_e, \mathcal{O}_{\mathbb{F}_e}(aC_e+bF))> 0$ if and only if $a < -2$ and $b < -e-2$.
\end{Remark}

\begin{Theorem} \cite[Theorem 2.1]{Coskun-Huizenga} \label{CH}
Let $L= \mathcal{O}_{\mathbb{F}_e}(aC_e+bF)$ be a line bundle on $\mathbb{F}_e$. Then
\begin{itemize}
\item [(i)] We have
\[\chi(L) = (a+1)(b+1)-e\frac{a(a+1)}{2}.\]
\item [(ii)] If $L \cdot F \geq -1$, then $h^2(\mathbb{F}_e, L)=0$.
\item [(iii)] If $L \cdot F \leq -1$, then $h^0(\mathbb{F}_e, L)=0$.
\item [(iv)] In particular, if $L \cdot F =-1$, then $L$ has no cohomology in any degree.
Now suppose $L \cdot F > -1$.  Then $h^2(\mathbb{F}_e, L)=0$, so either of the numbers $h^0(\mathbb{F}_e, L)$ or $h^1(\mathbb{F}_e, L)$ determine the cohomology of $L$. These can be determined as follows.
\item [(v)] If $L \cdot C_e \geq -1$, then $H^1(\mathbb{F}_e, L)=0$, and so $h^0(\mathbb{F}_e, L) = \chi(L)$.
\item [(vi)] If $L \cdot C_e < -1$, then $H^0(\mathbb{F}_e, L) \equiv H^0(\mathbb{F}_e, L(-C_e))$, and so the cohomology of $L$ can be determined indutively using $(iii)$ and $(v)$.
If $L \cdot F < -1$ then the cohomology of $L$ can be determined by Serre duality.
\end{itemize}
\end{Theorem}

\section{Segre invariant of rank $2$ vector bundles on Hirzebruch surfaces}

In this section, we study the Segre invariant for rank $2$ vector bundles  on the Hirzebruch surface $\mathbb{F}_e$, $e \geq 0$.  If $E$ is a rank $2$ vector bundle on $\mathbb{F}_e$ with first Chern class $c_1(E)=aC_e+bF $, there is a uniquely determined line bundle $\mathcal{O}_{\mathbb{F}_e}(a_1C_e+b_1F)$  such that 
\[c_1(E \otimes \mathcal{O}_{\mathbb{F}_e}(a_1C_e+b_1F) ) = -\alpha C_e -\beta F, \,\,\, \alpha,\beta \in \{0,1\}.\] 
Namely;
\[\mathcal{O}_{\mathbb{F}_e}(a_1C_e+b_1F) = \begin{cases}
     -\frac{c_1(E)}{2}, & \text{if $a $ and $b$ are even} \\
     -\frac{(a+1)C_e+bF}{2}, & \text{if $a$ is odd and $b$ is even} \\
     -\frac{aC_e+(b+1)F}{2}, & \text{if $a$ is even and $b$ is odd} \\
     -\frac{(a+1)C_e+(b+1)F}{2}, & \text{if $a$ is odd and $b$ is odd} \\
\end{cases}\]
Since $S_H(E) = S_H(E \otimes L)$ for any line bundle $L$ on $\mathbb{F}_e$, in the remainder of this section we assume that $E$ has  first Chern class  $c_1 = -\alpha C_e -\beta F$, $\alpha, \beta \in \{0,1\}$  and second Chern class $c_2$.

Let $H= C_e+(e+1)F$ be an ample divisor on $\mathbb{F}_e$, $e \geq 0$. The aim of this section is to give necessary  conditions for the existence of an $H-$stable vector bundle $E$ of rank $2$ with fixed Chern class $c_1$, $c_2$, and $S_H(E)=s$. The following results will be the key point for proving the main results of this section.

\begin{Remark}
    Let $H= C_e+(e+1)F$ be an ample divisor on $\mathbb{F}_e$, $e \geq 0$. Let $E$ be vector bundle on $\mathbb{F}_e$ with chern classes $c_1(E)= aC_e+bF$, then 
    \[\mu_H(E)= \frac{c_1(E)\cdot H}{rk (E)}=\frac{a+b}{rk (E)}.\]
\end{Remark}

Serre's construction  provides a method for constructing rank two vector bundles on a surface $X$ (see \cite[Chapter 5]{Huybrechts-Lehn} for more details):

\begin{Lemma}  \label{CB} 
Let $Z \subset X$ be a local complete intersection of codimension two in the projective non-singular surface $X$, and let $L$ and $M$ be line bundles on $X$. Then there exists an extension
\[0 \longrightarrow L \longrightarrow E \longrightarrow M \otimes I_Z \longrightarrow 0\]
such that $E$ is locally free if and only if the pair $(L^{\vee} \otimes M \otimes K_X,Z)$ satisfy the Cayley-Bacharach property:
\begin{align*}
 (CB) \,\,\,  & \text{if $Z' \subset Z$ is a sub-scheme with
 $l(Z')= l(Z)-1$ and }\\
 & \text{$s \in H^0(X, L^{\vee} \otimes M \otimes K_X)$  with  $s\vert _{Z'}=0$, then $s\vert_Z=0$}.
\end{align*}
\end{Lemma}

\begin{Remark}
\begin{itemize}
    \item [(i)] If $H^0(X,L^{\vee} \otimes M \otimes K_X)=0$, then $(CB)$ is satisfied for any $Z$;
    \item [(ii)] If $H^0(X,L^{\vee} \otimes M \otimes K_X)=l$, for a generic $Z$ with $l(Z) > l$, the sheaf $L^{\vee} \otimes M \otimes K_X \otimes I_Z$ has no non-trivial sections, and hence $(L^{\vee} \otimes M \otimes K_X,Z)$
    satisfies (CB).
\end{itemize}
\end{Remark}

\begin{Lemma}\label{HO}
    Let $Z \subset \mathbb{F}_e$, $e \geq 0$  be a local complete intersection of codimension two,  and let $a,b \in \mathbb{N}$,  $\alpha, \beta \in \{0,1\}$. Let $L = \mathcal{O}_{\mathbb{F}_e}(-a C_e - b F)$ be a line bundle on $\mathbb{F}_e$. Then, $H^0(L^\vee \otimes L^\vee \otimes \mathcal{O}_{\mathbb{F}_e}(-\alpha C_e - \beta F) \otimes K_{\mathbb{F}_e})=0$ whenever
    \[-\frac{\alpha+2}{2}<a \,\,\,\,\,\, \text{or} \,\,\,\,\,\, -\frac{\beta+e+2}{2} < b. \]
\end{Lemma}

\begin{Remark} \emph{
Let $H$  be an ample divisor on $\mathbb{F}_e$, $e \geq 0$, and let $E$ be a rank $2$ vector bundle on $\mathbb{F}_e$ such that $c_1(E) \cdot H \leq 0$. If $E$ is $H-$stable, then $H^0(\mathbb{F}_e,E)=0$.}
\end{Remark} 

\begin{Example}
Let $H = C_e+(e+1)F$  be an ample divisor on $\mathbb{F}_e$, $e \geq 0$.   Let $a, b \in \mathbb{N}$, $a\leq b$, $\alpha, \beta \in \{0,1\}$ and let $Z \subset \mathbb{F}_e$ be a local complete intersection of codimension two. By Lemma  \ref{CB} and Lemma \ref{HO} exists an extension 
\[0 \rightarrow \mathcal{O}_{\mathbb{F}_e}(-aC_e+bF)\rightarrow E \rightarrow \mathcal{O}_{\mathbb{F}_e}(-(a+\alpha)C_e+(b-\beta)F)\otimes I_Z \rightarrow 0\]
where $E$ is locally free with Chern class $c_1(E)=-\alpha C_e-\beta F$. Here, by Remark \ref{efample} follows that
\[h^0(\mathbb{F}_e, E) \leq h^0(\mathbb{F}_e, \mathcal{O}_{\mathbb{F}_e}(-aC_e+bF))+h^0(\mathbb{F}_e, \mathcal{O}_{\mathbb{F}_e}(-(a+\alpha)C_e+(b-\beta)F)\otimes I_Z)=0\]
Note that
\[\mu_{H}(E) = \frac{-(\alpha+\beta)}{2} \leq 0 \leq  c_1(\mathcal{O}_{\mathbb{F}_e}(-aC_e+bF)) \cdot H = -a+b.\]
Consequently $E$ is not $H-$stable.
\end{Example}

\begin{Lemma} \label{nosecciones}
Let $H: =C_e+(e+1)F$  be an ample divisor on $\mathbb{F}_e$, $e \geq 0$, and let $E$ be a rank $2$ vector bundle on $\mathbb{F}_e$ with Chern classes $c_1= -\alpha C_e- \beta F$,  $\alpha, \beta \in \{0,1\}$ and $c_2$. Let $a, b \in \mathbb{N}$.
\begin{enumerate}
    \item  Assume that $E$ can be written in the exact sequence  
\[0 \rightarrow \mathcal{O}_{\mathbb{F}_e}(-aC_e) \rightarrow E \rightarrow \mathcal{O}_{\mathbb{F}_e}((a-\alpha)C_e) \otimes I_Z \rightarrow 0,\]
with $a\neq 0$ and  $Z \subset \mathbb{F}_e$  is a local complete intersection of codimension two of length $l(Z)=c_2-a^2e+ae\alpha$.  If any of the following conditions is satisfied:
\begin{itemize}
    \item [(i)] $ae \leq e\alpha +1$ and $(a-\alpha+1)(2-e(a-\alpha))+ 2a^2e-2ae\alpha \leq 2c_2$.
    
    \item [(ii)] $e\alpha +1 < ae $ and $(a-m-\alpha+1)(2-e(a-m-\alpha))+ 2a^2e-2ae\alpha \leq 2c_2$, if there exists some smallest integer $0 < m < a+1-\alpha$ such that $ae \leq e(m+\alpha)+1$.
    
\end{itemize}
then $H^0(\mathbb{F}_e,E) = 0$.

\item Assume that $E$ can be written in the exact sequence  
\[0 \rightarrow \mathcal{O}_{\mathbb{F}_e}(-bF) \rightarrow E \rightarrow \mathcal{O}_{\mathbb{F}_e}((b-\beta)F)\otimes I_Z \rightarrow 0.\]
with $b\neq 0$ and  $Z \subset \mathbb{F}_e$  is a local complete intersection of codimension two of length $l(Z)=c_2$. If $b-\beta+1 \leq c_2$, then $H^0(\mathbb{F}_e,E) = 0$.

\item Assume that $E$ can be written in the exact sequence \[0 \rightarrow \mathcal{O}_{\mathbb{F}_e}(-aC_e-bF) \rightarrow E \rightarrow \mathcal{O}_{\mathbb{F}_e}((a-\alpha)C_e+(b-\beta)F) \otimes I_Z\]
with $a, b \neq 0$ and  $Z \subset \mathbb{F}_e$  is a local complete intersection of codimension two of length $l(Z)=c_2-a^2e+2ab+ae\alpha-a\beta-b\alpha$. If any of the following conditions is satified
\begin{itemize}
    \item [(i)] $ae+\beta \leq e\alpha+b+1$ and $(a+1-\alpha)(2(b+1-\beta)-e(a-\alpha))+2(a^2e-2ab-ae\alpha+a\beta+b\alpha) \leq 2c_2$  

\item [(ii)]  $ e\alpha+b+1 < ae+\beta$ and $(a+1-m-\alpha)(2(b+1-\beta)-e(a-m-\alpha))+2(a^2e-2ab-ae\alpha+a\beta+b\alpha) \leq 2c_2$, if there exists some smallest integer $0 < m < a+1-\alpha$ such that $ea+\beta \leq e(m+\alpha)+b+1$. 

\end{itemize}
then $H^0(\mathbb{F}_e,E) = 0$.
\end{enumerate}
\end{Lemma}

\begin{proof}
 The proof follows directly of Theorem \ref{CH}.
\end{proof}

We can now state the principal results of this section. Let $H:= C_e +(e+1)F$  be an ample divisor on $\mathbb{F}_e$, $e \geq 0$.  We give neccesary conditons for the existence of an $H-$stable vector bundle of rank $2$, with Chern classes $c_1$, $c_2$ where $c_1= -\alpha C_e-\beta F$, $\alpha, \beta \in \{0,1\}$, and a fixed value $S_H(E):=s > 0$.  Note that $s$ can be written in one of the following ways (see Remark \ref{NotUnique}.):

\begin{itemize}
    \item [(i)] $s = 2(a-b)-\alpha-\beta$.
    
    \item [(ii)] $s= 2(b-a)-\alpha-\beta$.
    
    \item [(iii)] $s = 2(a+b)-\alpha-\beta$.
\end{itemize}
for some $a,b \geq 0$.

\begin{Theorem}\label{TP1}
Let $H=C_e+(e+1)F$   be an ample divisor on $\mathbb{F}_e$, $e \geq 0$ and let \\
$\alpha, \beta \in \{0,1\}$, $c_2 \geq 2$, $a,b \in \mathbb{N}$, $a  \neq 0$ and set $s:=2(a-b)-\alpha-\beta >0$.  Let $Z \subset \mathbb{F}_e$ be a local complete intersection of codimension two of lenght $l(Z)=c_2-a^2e+ae\alpha-2ab -a\beta+b\alpha > 0$.  If \begin{equation}\label{TP1a1}
    \begin{aligned}  
        h^0(\mathbb{F}_e,  \mathcal{O}_{\mathbb{F}_e}((a-\alpha)C_e+ & (2(a-b)-\alpha-\beta-1)F \otimes I_Z)=  \\ & h^0(\mathbb{F}_e, \mathcal{O}_{\mathbb{F}_e}((2a-\alpha-1)C_e+(a-2b-\beta-1)F \otimes I_Z)=0, 
    \end{aligned}
\end{equation} 
then there exists an $H-$stable vector bundle $E$  with Chern classes $c_1(E) = -\alpha C_e-\beta F$, $c_2$ and $S_H(E)=s$. 
Furthermore, $E$ can be written in the following exact sequence:
\[0 \rightarrow \mathcal{O}_{\mathbb{F}_e}(-aC_e+bF) \rightarrow E \rightarrow \mathcal{O}_{\mathbb{F}_e}((a-\alpha)C_e-(b+\beta)F) \otimes I_Z \rightarrow 0\]
with  $\mathcal{O}_{\mathbb{F}_e}(-aC_e+bF) \subset E$ maximal.
\end{Theorem}

\begin{Theorem} \label{TP2}
Let $H=C_e+(e+1)F$   be an ample divisor on $\mathbb{F}_e$, $e \geq 0$ and let $\alpha, \beta \in \{0,1\}$, $c_2 \geq 2$, $a,b \in \mathbb{N}$, $b  \neq 0$ and set $s:=2(b-a)-\alpha-\beta >0$.  Let $Z \subset \mathbb{F}_e$ be a local complete intersection of codimension two of lenght $l(Z)=c_2-a^2e-ae\alpha-2ab +a\beta-b\alpha > 0$.  If   
    \begin{equation} \label{TP2a2} 
    \begin{aligned}
        h^0(\mathbb{F}_e,  \mathcal{O}_{\mathbb{F}_e}& (2(b-a)-\alpha-\beta-1)C_e+  (b-\beta)F \otimes I_Z)=  \\ & h^0(\mathbb{F}_e, \mathcal{O}_{\mathbb{F}_e}((b-2a-\alpha-1)C_e+(2b-\beta-1)F \otimes I_Z)=0, 
    \end{aligned}
    \end{equation}
then there exists an $H-$stable vector bundle $E$  with Chern classes $c_1(E) = -\alpha C_e-\beta F$, $c_2$ and $S_H(E)=s$. 
Furthermore, $E$ can be written in the following exact sequence:
\[0 \rightarrow \mathcal{O}_{\mathbb{F}_e}(aC_e-bF) \rightarrow E \rightarrow \mathcal{O}_{\mathbb{F}_e}(-(a+\alpha)C_e+(b-\beta)F) \otimes I_Z \rightarrow 0\]
with  $\mathcal{O}_{\mathbb{F}_e}(aC_e-bF) \subset E$ maximal. 
\end{Theorem}

\begin{Theorem} \label{TP3}
 Let $H=C_e+(e+1)F$,   be an ample divisor on $\mathbb{F}_e$, $e \geq 0$ and let $\alpha, \beta \in \{0,1\}$, $c_2 \geq 2$, $a,b \in \mathbb{N}-\{0\}$ and set $s:=2(a+b)-\alpha-\beta >0$.  Let $Z \subset \mathbb{F}_e$ be a local complete intersection of codimension two of lenght $l(Z)=c_2-a^2e+ae\alpha+2ab -a\beta-b\alpha >0$ satisfying any of the conditions of item (3), Lemma \ref{nosecciones}.   If 
 the pair 
\begin{equation}
(\mathcal{O}_{\mathbb{F}_e}((2a-2)C_e + (2b-e-2-\beta)F), I_Z)
\end{equation}
satifies the property of Cayley-Bacharach (See Theorem \ref{CB}) and 
    \begin{equation} \label{TP3a3} 
    \begin{aligned}
        h^0(\mathbb{F}_e,  \mathcal{O}_{\mathbb{F}_e}&  (2(a+b)-\alpha-\beta-1)C_e+  (b-\beta)F \otimes I_Z)= \\ &  h^0(\mathbb{F}_e, \mathcal{O}_{\mathbb{F}_e}((a-\alpha)C_e+  (2(a+b)-\alpha-\beta-1)F \otimes I_Z)= \\ & h^0(\mathbb{F}_e, \mathcal{O}_{\mathbb{F}_e}((2a-\alpha-1)C_e+  (a+2b-\beta-1)F \otimes I_Z) = \\ &
        h^0(\mathbb{F}_e, \mathcal{O}_{\mathbb{F}_e}((2a+b-\alpha-1)C_e+  (2b-\beta-1)F \otimes I_Z)=0, 
    \end{aligned}
    \end{equation}
then there exists an $H-$stable vector bundle $E$  with Chern classes $c_1(E) = -\alpha C_e-\beta F$, $c_2$ and $S_H(E)=s$. 
Furthermore, $E$ can be written in the following exact sequence:
\[0 \rightarrow \mathcal{O}_{\mathbb{F}_e}(-aC_e-bF) \rightarrow E \rightarrow \mathcal{O}_{\mathbb{F}_e}((a-\alpha)C_e+(b-\beta)F) \otimes I_Z \rightarrow 0\]
with  $\mathcal{O}_{\mathbb{F}_e}(-aC_e-bF) \subset E$ maximal. 
\end{Theorem}

\begin{proof}
\textbf{Proof of Theorem \ref{TP1}, Theorem \ref{TP2},  Theorem \ref{TP3}}.  We only prove Theorem \ref{TP1}, since the proofs  of Theorems \ref{TP2},  \ref{TP3} are quite analogous.    
Assume that 
$h^0(\mathbb{F}_e,  \mathcal{O}_{\mathbb{F}_e}((a-\alpha)C_e+(2(a-b)-\alpha-\beta-1)F \otimes I_Z)=  h^0(\mathbb{F}_e, \mathcal{O}_{\mathbb{F}_e}((2a-\alpha-1)C_e+(a-2b-\beta-1)F \otimes I_Z)=0$. Let $H=C_e+(e+1)F$ be an ample divisor on $\mathbb{F}_e$, $e \geq 0$ and let $\alpha, \beta \in \{0,1\}$.  Let $Z \subset \mathbb{F}_e$ be a local complete intersection of codimension two and length $l(Z) = c_2-a^2e+ae\alpha-2ab-a\beta+b\alpha$.  By Lemma \ref{CBH}, the pair $(\mathcal{O}_{\mathbb{F}_e}((2a-\alpha-2)C_e-(2b+e+\beta+2)F), Z)$ satisfies  Cayley-Bacharach property. Therefore, there exists an extension 
\begin{equation} \label{E1}
0 \rightarrow \mathcal{O}_{\mathbb{F}_e}(-aC_e+bF) \rightarrow E \rightarrow \mathcal{O}_{\mathbb{F}_e}((a-\alpha)C_e-(b+\beta) F) \otimes I_Z \rightarrow 0
\end{equation}
where $E$ is  locally free with Chern classes $c_1(E)=-\alpha C_e-  \beta F$ and $c_2(E)=c_2$. Moreover,  we have that $H^0(\mathbb{F}_e,E)=0$, because $h^0(\mathbb{F}_e, \mathcal{O}_{\mathbb{F}_e}(-aC_e+bF))= h^0(\mathbb{F}_e, \mathcal{O}_{\mathbb{F}_e}((a-\alpha)C_e-(b+\beta)F))=0$, whenever $b\neq 0$ or $\beta =1$. If $b=0$ and $\beta =0$, then $H^0(\mathbb{F}_e, \mathcal{O}_{\mathbb{F}_e}((a-\alpha)C_e)\otimes I_Z)=0$ because $Z$ is not contained in any curve of bidegree $(a-\alpha, 2(a-b)-\alpha-\beta-1)$. 

Our next goal is to show 
that in the exact sequence $(\ref{E1})$, the line bundle $\mathcal{O}_{\mathbb{F}_e}(-aC_e+bF)$ is maximal.

 Let $L$ be a subline bundle of $E$ such that 
\[c_1(\mathbb{F}_e, \mathcal{O}_{\mathbb{F}_e}(-aC_e+bF))\cdot H < c_1(L) \cdot H.\]
We must have $H^0(\mathbb{F}_e, \mathcal{O}_{\mathbb{F}_e}(-aC_e+bF) \otimes L^\vee)=0$,  otherwise we would have  a morphism $L \rightarrow \mathcal{O}_{\mathbb{F}_e}(-aC_e+bF)$ where $c_1(L) \cdot H \leq c_1(\mathbb{F}_e, \mathcal{O}_{\mathbb{F}_e}(-aC_e+bF))\cdot H $. Since  $H^0(\mathbb{F}_e,E)=0$, we have that $L$ can be written in one of the following ways;
\[\mathcal{O}_{\mathbb{F}_e}(-a_1C_e+b_1F), \,\,\,\, \mathcal{O}_{\mathbb{F}_e}(a_1C_e-b_1F), \,\,\, \mathcal{O}_{\mathbb{F}_e}(-a_1C_e-b_1F)\]
with $0 \leq a_1, b_1$ both non-zero.    We consider all possiblities for $L$, and we will prove that $L$ is not a subline bundle of $E$. Indeed,

\begin{itemize}
\item [(i)] Suppose that $L=\mathcal{O}_{\mathbb{F}_e}(-a_1C_e+b_1F)$, $a_1 \neq 0$ and
\begin{equation} \label{q1}
-a+b <  -a_1+b_1.
\end{equation} From the extension (\ref{E1}), we get
\[\begin{aligned}
0 & \rightarrow \mathcal{O}_{\mathbb{F}_e}((a_1-a)C_e+(b-b_1)F)   \rightarrow E \otimes \mathcal{O}_{\mathbb{F}_e}(a_1C_e-b_1F) \\ &\rightarrow 
 \mathcal{O}_{\mathbb{F}_e}((a_1+a-\alpha)C_e-(b+b_1+\beta)F)  \otimes I_Z \rightarrow 0
\end{aligned}\]
where
$H^0(\mathbb{F}_e, \mathcal{O}_{\mathbb{F}_e}((a_1-a)C_e+(b-b_1)F) )= H^0(\mathbb{F}_e,  \mathcal{O}_{\mathbb{F}_e}((a_1+a-\alpha)C_e-(b+b_1+\beta)F)  \otimes I_Z)=0$ whenever $b+b_1+\beta> 0$.  If $b=b_1=\beta=0$, from inequality (\ref{q1}) we have $a_1<a$. Since $h^0(\mathbb{F}_e, \mathcal{O}_{\mathbb{F}_e}((2a-\alpha-1)C_e)+(a-2b-\beta-1) \otimes I_Z)=0$, it follows that $Z$ is not contained in any curve of bidegree $(2a-\alpha-1, a-2b-\beta-1)$. Therefore $Z$ is not contained in any curve of bidegree $(a+a_1-\alpha, 0)$. Hence, $H^0(\mathbb{F}_e, E \otimes \mathcal{O}_{\mathbb{F}_e}(a_1C_e-b_1F))=0.$

\item [(ii)] Suppose that $L = \mathcal{O}_{\mathbb{F}_e}(a_1C_e-b_1F)$, $b_1 \neq 0$ and 
\begin{equation} \label{q2}
-a+b <  a_1-b_1.
\end{equation}
From the extension (\ref{E1}), we get
\[\begin{aligned}
0 & \rightarrow \mathcal{O}_{\mathbb{F}_e}(-(a_1+a)C_e+(b+b_1)F)  \rightarrow  E \otimes \mathcal{O}_{\mathbb{F}_e}(-a_1C_e+b_1F) \\  &\rightarrow 
 \mathcal{O}_{\mathbb{F}_e}((a-a_1-\alpha)C_e-(b-b_1+\beta)F)  \otimes I_Z \rightarrow 0
\end{aligned}\]
where $H^0(\mathbb{F}_e, \mathcal{O}_{\mathbb{F}_e}(-(a_1+a)C_e+(b+b_1)F) )= 0$.  Note that
$H^0(\mathbb{F}_e, \mathcal{O}_{\mathbb{F}_e}((a-a_1-\alpha)C_e+(b_1-b-\beta)F)) \neq 0$, if and only if $a_1 \leq a-\alpha$ and $b+\beta \leq b_1$.  Assume $a_1 \leq a-\alpha$ (or $b+\beta \leq b_1$). From inequality (\ref{q2}), we have
\[0\leq b_1-b-\beta <  2(a-b)-\alpha-\beta.\]
Since $h^0(\mathbb{F}_e,  \mathcal{O}_{\mathbb{F}_e}((a-\alpha)C_e+(2(a-b)-\alpha-\beta-1) \otimes I_Z)=0$ , it follows that $ Z$ is not contained in any curve of bidegree $(a-\alpha, 2(a-b)-\alpha-\beta-1)$. Hence, $Z$ is not contained in any curve of bidegree $(a-a_1-\alpha, b-b_1-\beta)$ and
$H^0(\mathbb{F}_e, \mathcal{O}_{\mathbb{F}_e}((a-a_1-\alpha)C_e-(b_1-b-\beta)F) \otimes I_Z)=0$. Therefore, $H^0(\mathbb{F}_e, E \otimes \mathcal{O}_{\mathbb{F}_e}(-a_1C_e+b_1F))=0$.

\item [(iii)] Suppose that $L = \mathcal{O}_{\mathbb{F}_e}(-a_1C_e-b_1F)$, $a_1, b_1$ both  non-zero, and
\begin{equation} \label{q3}
-a+b <  -a_1-b_1.
\end{equation} 
From the extension (\ref{E1}), we get 
\[\begin{aligned}
0 & \rightarrow \mathcal{O}_{\mathbb{F}_e}((a_1-a)C_e+(b+b_1)F) \rightarrow  E \otimes \mathcal{O}_{\mathbb{F}_e}(a_1C_e+b_1F) \\ &   \rightarrow 
 \mathcal{O}_{\mathbb{F}_e}((a+a_1-\alpha)C_e-(b-b_1+\beta)F)  \otimes I_Z \rightarrow 0
\end{aligned}\]
where $H^0(\mathbb{F}_e, \mathcal{O}_{\mathbb{F}_e}((a_1-a)C_e+(b+b_1)F)=0$ which implies $a_1 < a$.  Note that
$H^0(\mathbb{F}_e,  \mathcal{O}_{\mathbb{F}_e}((a+a_1-\alpha)C_e-(b-b_1+\beta)F)) \neq 0$, if and only if $0\leq a_1+a -\alpha$ and $b+\beta \leq b_1$.   From inequality (\ref{q3}), we have
\[0\leq b_1-b-\beta <  a-2b-\beta.\]
Since $h^0(\mathbb{F}_e,  \mathcal{O}_{\mathbb{F}_e}((2a-\alpha-1)C_e+(a-2b-\beta-1) \otimes I_Z)=0$ , it follows that $ Z$ is not contained in any curve of bidegree $(2a-\alpha-1, a-2b-\beta-1)$.  Hence $Z$ is not contained in any curve of bidegree $(a+a_1-\alpha, b_1-b-\beta)$ and  
$H^0(\mathbb{F}_e, \mathcal{O}_{\mathbb{F}_e}((a+a_1)C_e-(b-b_1+\beta)F) \otimes I_Z)=0$. Therefore, $H^0(\mathbb{F}_e, E \otimes \mathcal{O}_{\mathbb{F}_e}(a_1C_e+b_1F))=0$.
\end{itemize}
From (i), (ii) and (iii), it follows that $H^0(\mathbb{F}_e, E \otimes L^\vee)=0$. Therefore $L$ is not a subline bundle of $E$ and $\mathcal{O}_{\mathbb{F}_e}(-aC_e+bF)$ is maximal as we desired.

Finally, we proceed to show that $E$ is $H-$stable.  Let $L$ be a subline bundle of $E$, since that $\mathcal{O}_{\mathbb{F}_e}(-aC_e+bF)$ is maximal, we have
\[c_1(L) \cdot H \leq c_1(\mathcal{O}_{\mathbb{F}_e}(-aC_e+bF)) \cdot H = -a+b < \mu_{H}(E) = \frac{-\alpha-\beta }{2}, \]
because $s=2(a-b)-\alpha-\beta >0$. Therefore,  $E$ is $H-$stable which completes the proof.


\end{proof}

The following proposition gives numerical conditions under which the hypothesis (\ref{TP1a1}) (resp.,  hypothesis (\ref{TP2a2}), and   (\ref{TP3a3}))  of Theorem \ref{TP1}  (resp., Theorem \ref{TP2}, and Theorem \ref{TP3}) is satisfied.

\begin{Proposition} \label{PP}
\begin{enumerate}
    \item Let $\alpha, \beta \in \{0,1\}$, $c_2 \geq 2$, $a,b \in \mathbb{N}$, $a\neq 0$. Let $Z \subset \mathbb{F}_e$, $e\geq 0$ be a local complete intersection of codimension two of lenght $l(Z)=c_2-a^2e+ae\alpha-2ab -a\beta+b\alpha > 0$.  If any of the following conditions is satisfied

\begin{itemize}
\item [(i)]  $e(2a-\alpha-1) \leq 2(a-b)-\alpha-\beta$ and $(2a-\alpha)(2(2a-2b-\alpha-\beta)-e(2a-\alpha-1))+ 2(a^2e-ae\alpha+2ab+a\beta-b\alpha) \leq 2c_2$,

\item [(ii)]  $  2(a-b)-\alpha-\beta < e(2a-\alpha-1)$ and $(2a-m-\alpha)(2(2a-2b-\alpha-\beta)-e(2a--m-\alpha-1))+ 2(a^2e-ae\alpha+2ab+a\beta-b\alpha) \leq 2c_2$, if there is some smallest integer $0< m < 2a-\alpha$ such that $2(2a-m-\alpha-1) \leq 2a-2b-\alpha-\beta$, 

\end{itemize}
then 
\begin{align*}
        h^0(\mathbb{F}_e,  \mathcal{O}_{\mathbb{F}_e}((a-\alpha)C_e+ & (2(a-b)-\alpha-\beta-1)F \otimes I_Z)=  \\ & h^0(\mathbb{F}_e, \mathcal{O}_{\mathbb{F}_e}((2a-\alpha-1)C_e+(a-2b-\beta-1)F \otimes I_Z)=0.
    \end{align*}
    
    \item Let $\alpha, \beta \in \{0,1\}$, $c_2 \geq 2$, $a,b \in \mathbb{N}$, $b \neq 0$. Let $Z \subset \mathbb{F}_e$, $e\geq 0$ be a local complete intersection of codimension two of lenght $l(Z)=c_2-a^2e-ae\alpha-2ab +a\beta-b\alpha >0$.  If any of the following conditions is satisfied

\begin{itemize}
\item [(i)]  $e(2b-2a-\alpha-\beta) \leq e+2b-beta$ and $(2b-2a-\alpha-\beta)(2(2b-\beta)-e(2b-2a-\alpha-\beta-1))+2(a^2e+ae\alpha+2ab-a\beta+b\alpha) \leq 2c_2$,

\item [(ii)]  $ e+2b-beta <   e(2b-2a-\alpha-\beta) $ and $(2b-2a-m-\alpha-\beta)(2(2b-\beta)-e(2b-2a-m-\alpha-\beta-1))+2(a^2e+ae\alpha+2ab-a\beta+b\alpha) \leq 2c_2$, if there is some smallest integer $0< m < 2b-2a-\alpha-\beta$ such that $e(2b-2a-\alpha-\beta) \leq e(m+1)+2b-\beta$, 

\end{itemize}
then \[
\begin{aligned}
        h^0(\mathbb{F}_e,  \mathcal{O}_{\mathbb{F}_e}& (2(b-a)-\alpha-\beta-1)C_e+  (b-\beta)F \otimes I_Z)=  \\ & h^0(\mathbb{F}_e, \mathcal{O}_{\mathbb{F}_e}((b-2a-\alpha-1)C_e+(2b-\beta-1)F \otimes I_Z)=0, 
    \end{aligned}\]
    
\item Let $\alpha, \beta \in \{0,1\}$, $c_2 \geq 2$, $a,b \in \mathbb{N}$, $a,b \neq 0$. Let $Z \subset \mathbb{F}_e$, $e\geq 0$ be a local complete intersection of codimension two of lenght $l(Z)=c_2-a^2e+ae\alpha+2ab -a\beta-b\alpha >0$.  If any of the following conditions is satisfied

\begin{itemize}
\item [(i)]  $(e-1)(2a+2b-\alpha-\beta) \leq e$ and $(2a+2b-\alpha-\beta)(e+(2-e)(2a+2b-\alpha-\beta))+2(a^2e-ae\alpha-2ab+a\beta-b\alpha) \leq 2c_2$,

\item [(ii)]  $  e < (e-1)(2a+2b-\alpha-\beta)$ and $(2a+2b-\alpha-\beta-m)(e(m+1)+(2-e)(2a+2b-\alpha-\beta))+2(a^2e-ae\alpha-2ab+a\beta-b\alpha) \leq 2c_2$, if there is some smallest integer $0< m < 2a+2b-\alpha-\beta$ such that $(e-1)(2a+2b-\alpha-\beta) \leq (m+1)e$, 

\end{itemize}
then \[
\begin{aligned}
        h^0(\mathbb{F}_e,  \mathcal{O}_{\mathbb{F}_e}&  (2(a+b)-\alpha-\beta-1)C_e+  (b-\beta)F \otimes I_Z)= \\ &  h^0(\mathbb{F}_e, \mathcal{O}_{\mathbb{F}_e}((a-\alpha)C_e+  (2(a+b)-\alpha-\beta-1)F \otimes I_Z)= \\ & h^0(\mathbb{F}_e, \mathcal{O}_{\mathbb{F}_e}((2a-\alpha-1)C_e+  (a+2b-\beta-1)F \otimes I_Z) = \\ &
        h^0(\mathbb{F}_e, \mathcal{O}_{\mathbb{F}_e}((2a+b-\alpha-1)C_e+  (2b-\beta-1)F \otimes I_Z)=0, 
    \end{aligned}\]
\end{enumerate}
\end{Proposition}

\begin{proof}
We only prove (1), since the proof of (2) and (3) are quite analogous.  Let $Z \subset \mathbb{F}_e$, $e\geq 0$ be a local complete intersection of codimension two of lenght $l(Z)=c_2-a^2e+ae\alpha-2ab -a\beta+b\alpha >0$ and let $L:=\mathcal{O}_{\mathbb{F}_e}((2a-\alpha-1)C_e+(2a-2b-\alpha-\beta-1)F)$.  From Theorem \ref{CH}, it follows that $H^0(\mathbb{F}_e, L \otimes I_Z)=0$, If any of the conditions (i) and (ii)  are satisfied. Therefore, $Z$ is not contained in any curve of bidegree $(2a-\alpha-1, 2a-2b-\alpha-\beta-1)$. Hence, \begin{align*}
        h^0(\mathbb{F}_e,  \mathcal{O}_{\mathbb{F}_e}((a-\alpha)C_e+ & (2(a-b)-\alpha-\beta-1)F \otimes I_Z)=  \\ & h^0(\mathbb{F}_e, \mathcal{O}_{\mathbb{F}_e}((2a-\alpha-1)C_e+(a-2b-\beta-1)F \otimes I_Z)=0,
    \end{align*}
    as we desired  
\end{proof}

\section{ Stratification of the moduli space $M_{\mathbb{F}_e}(2,c_1,c_2)$ according to the invariant $S_H$}

Let $H:=C_e+ (e+1)F$ be an ample divisor on $\mathbb{F}_e$, $e \geq 0$. In this section we use the Segre invariant to induce a
stratification of the moduli space  $M_{\mathbb{F}_e,H}(2;c_1,c_2)$
of $H-$stable vector bundles of rank  $2$, Chern classes $c_1$,
$c_2$ on $\mathbb{F}_e$, $e \geq 0$. Working locally in the \'etale
topology, we can assume that there is a family $\mathcal{E}$
parameterized by $M_{\mathbb{F}_e,H}(2;c_1,c_2)$.

From Theorem \ref{semicontinuous}, the function $S:
M_{\mathbb{F}_e, H}(2;c_1,c_2) \longrightarrow \mathbb{Z}$ induces a
stratification of $M_{\mathbb{F}_e,H}(2;c_1,c_2)$ into locally
closed subsets
\[M_{\mathbb{F}_e,H}(2;c_1,c_2;s) := \{E \in M_{\mathbb{F}_e, H}(2;c_1,c_2) : S_H(E) = s\}\]
according to the value of $s$. 

In the following theorems, we stablish a lower bound of the dimension of the stratum $M_{\mathbb{F}_e,H}(2;c_1,c_2;s)$ whenever  it is non-empty.

\begin{Theorem} \label{dimmax1}
Let $H=C_e+(e+1)F$  be an ample divisor on $\mathbb{F}_e$, $e \geq 0$ and let $ \beta \in \{0,1\}$, $c_2 \geq 2$, $s, l_1,l_2, l_3 > 0$.  Suppose that there exist $a_i,b_i \in \mathbb{N}$,  for $i=1,2,3$ such that 
\[
\begin{aligned}
    s  = 2(a_1-b_1)-\beta & = 2(b_2-a_2)-\beta = 2(a_3+b_3)-\beta, \,\,\, a_1, b_2 ,a_3,b_3 \neq 0 \\
    l_1 &= c_2-a_1^2e-2a_1b_1-a_1\beta \\
    l_2 &= c_2-a_2^2e-2a_2b_2 +a_2\beta \\
    l_3 &= c_2-a_3^2e+2a_3b_3-a_3\beta
\end{aligned}\] 
Let $D_1:= a_1C_e-b_1F$ be a divisor on $\mathbb{F}_e$, and let $Z_1 \subset \mathbb{F}_e$ be a local complete intersection of codimension two of lenght $l_1$,   satisfiying the hypothesis of  Theorem \ref{TP1} taking $\alpha =0$. Let $D_2:= -a_2C_e+b_2F$ be a divisor on $\mathbb{F}_e$, and  let $Z_2 \subset \mathbb{F}_e$ be a local complete intersection of codimension two of lenght $l_2$,   satisfiying the hypothesis of  Theorem \ref{TP2} taking $\alpha =0$.  Let $D_3:= a_3C_e+b_3F$ be a divisor on $\mathbb{F}_e$, and let $Z_3 \subset \mathbb{F}_e$ be a local complete intersection of codimension two of lenght $l_3$,   satisfiying the hypothesis of  Theorem \ref{TP3} taking $\alpha =0$. Then, the stratum $M_{\mathbb{F}_e,H}(2;-\beta F,c_2;s)$ is non-empty and it has an irredubible subvariety of dimension
\[\begin{aligned}
     \max_{a_i,b_i} \{ 3l+ \frac{1}{2} (2a_1-1)(4b_1+2a_1e+2\beta+2),  & \,\,  3l+ \frac{1}{2}(2a_2+1)(4b_2-2a_2e-2\beta-4e-2), \\ & 3l+ \frac{1}{2}(2a_3-1)(2\beta+2a_3e-4b_3+2) \}
\end{aligned}\]
where the maximum is taken over all numbers $a_i,b_i$ that satisfies  the hypothesis of theorem. Moreover, 
\[\begin{aligned}
   \dim\, & M_{\mathbb{F}_e}(2; -\beta F,c_2;s) \geq 
    \max_{a_i,b_i} \{ 3l+ \frac{1}{2} (2a_1-1)(4b_1+2a_1e+2\beta+2), \\ & 3l+ \frac{1}{2}(2a_2+1)(4b_2-2a_2e-2\beta-4e-2),  \,\, 3l+ \frac{1}{2}(2a_3-1)(2\beta+2a_3e-4b_3+2) \}
\end{aligned}\]  
\end{Theorem}

\begin{Theorem} \label{dimmax2}
Let $H=C_e+(e+1)F$  be an ample divisor on $\mathbb{F}_e$, $e \geq 0$ and let $ \beta \in \{0,1\}$, $c_2 \geq 2$, $s > 0$.  Suppose that there exist $a_i,b_i \in \mathbb{N}$, $b_i \neq 0$ for $i=1,2,3$ such that  \[\begin{aligned}
    s= 2(a_1-b_1)-\beta-1 & = 2(b_2-a_2)-\beta-1 = 2(a_3+b_3)-\beta-1 \,\,\, a_2, b_1 ,a_3,b_3 \neq 0 \\
    l_1 & = c_2-a_1^2e+a_1e-2a_1b_1+b_1-a_1\beta\\
    l_2 & =  c_2-a_2^2e-a_2e-2a_2b_2-b_2 +a_2\beta \\
    l_3 & = c_2-a_3^2e+a_3e+2a_3b_3-b_3 -a_3\beta\\
\end{aligned}\] 
Let $D_1:= a_1C_e-b_1F$ be a divisor on $\mathbb{F}_e$, and  let $Z_1 \subset \mathbb{F}_e$ be a local complete intersection of codimension two and lenght $l_1$,   satisfiying the hypothesis of  Theorem \ref{TP1} taking $\alpha =1$. Let $D_2:= -a_2C_e+b_2F$ be a divisor on $\mathbb{F}_e$, and  let $Z_2 \subset \mathbb{F}_e$ be a local complete intersection of codimension two and lenght $l_2$,   satisfiying the hypothesis of  Theorem \ref{TP2} taking $\alpha =1$. Let $D_3:= a_3C_e+b_3F$ be a divisor on $\mathbb{F}_e$,   and let $Z_3 \subset \mathbb{F}_e$ be a local complete intersection of codimension two and lenght $l_3 $,   satisfiying the hypothesis of  Theorem \ref{TP3} taking $\alpha =1$. Then, the stratum $M_{\mathbb{F}_e,H}(2;- C_e-\beta F,c_2;s)$ is non-empty and it has an irredubible subvariety of dimension
\[\begin{aligned}
\max_{a_i,b_i}\{
   & 3l+ \frac{1}{2} (2a_1-2)(4b_1+2a_1e- e+2\beta+2), \\
   & 3l+ \frac{1}{2}(2a_2+2)(4b_2-2a_2e- e-2\beta-4e-2), \,\,  3l+ \frac{1}{2}(2a_3-2)(2\beta+2a_3e -4b_3+1) \}
   \end{aligned}
\]
where the maximum is taken over all numbers $a_i,b_i$ that satisfies  the hypothesis of theorem. Moreover, 
\[\begin{aligned}
   \dim\, & M_{\mathbb{F}_e}(2;- C_e -\beta F,c_2;s) \geq 
    \max_{a_i,b_i}\{
   3l+ \frac{1}{2} (2a_1-2)(4b_1+2a_1e- e+2\beta+2), \\
   & 3l+ \frac{1}{2}(2a_2+2)(4b_2-2a_2e- e-2\beta-4e-2), \,\, 3l+ \frac{1}{2}(2a_3-2)(2\beta+2a_3e -4b_3+1) \}
   \end{aligned}
\]
\end{Theorem}

The proofs of  Theorems \ref{dimmax1} and  \ref{dimmax2} make use of the following results:

\begin{Proposition} \label{dimensionstratum}
Let $H=C_e+(e+1)F$  be an ample divisor on $\mathbb{F}_e$, $e \geq 0$ and let $\alpha, \beta \in \{0,1\}$, $c_2 \geq 2$, $s,l > 0$.  Suppose that there exist $a,b \in \mathbb{N}$, $a \neq 0$ such that $s= 2(a-b)-\alpha-\beta$ and $l= c_2-a^2e+ae\alpha-2ab-a\beta+b\alpha >0$.  Let $D:= aC_e-bF$ be a divisor on $\mathbb{F}_e$ and let $Z \subset \mathbb{F}_e$ be a local complete intersection of codimension two and lenght $l$
 satisfiying the hypothesis of  Theorem \ref{TP1}.
   Then, the stratum $M_{\mathbb{F}_e,H}(2;-\alpha C_e-\beta F,c_2;s)$ is non-empty and it has an irredubible subvariety of dimension
   \[\max_{a,b} \left\{3l+  \frac{(2a-\alpha-1)(4b+2ae-\alpha e+2\beta+2)}{2} \right\}\]
where the maximum is taken over all pairs (a,b) satisfying the hypothesis of the proposition. Moreover, 
\[
 \dim\, M_{\mathbb{F}_e}(2;-\alpha C_e -\beta F,c_2;s) \geq  \max_{a,b} \left\{3l+  \frac{(2a-\alpha-1)(4b+2ae-\alpha e+2\beta+2)}{2} \right\}
\]
\end{Proposition}

\begin{proof}
We only prove the case $\alpha=0$, the proof of the case $\alpha=1$ follows by similar arguments.
 Let $D:=aC_e-bF$, $(a\neq 0)$  be a divisor  on $\mathbb{F}_e$, $e \geq 0$ such that $a,b$ satifies $s= 2(a-b)-\beta>0$ and  $l = c_2-a^2e-2ab-a\beta >0$.  Let $Z \subset \mathbb{F}_e$ be a local complete intersection of codimension two of lenght $l$ satifying the hypothesis of Theorem \ref{TP1} taking $\alpha=0$.  Let
$Hilb^{l}(\mathbb{P}^2)$ be the Hilbert scheme of zero-dimensional
subschemes of length $l$ and let
$\mathcal{I}_{\mathcal{Z}_l}$ be the ideal sheaf of the universal
subscheme $\mathcal{Z}_l$ in $\mathbb{F}_e \times
Hilb^{l}(\mathbb{F}_e)$.
Let $p_1$, $p_2$ be the projections of $\mathbb{F}_e \times
Hilb^{l}(\mathbb{F}_e)$ on $\mathbb{F}_e$ and
$Hilb^{l}(\mathbb{F}_e)$ respectively. Consider on $\mathbb{F}_e
\times Hilb^{l}(\mathbb{F}_e)$ the sheaf
$\mathcal{H}om(p_{1}^*\mathcal{O}_{\mathbb{F}_e}(D - \beta F) \otimes
\mathcal{I}_{\mathcal{Z}_l},
p_{1}^*\mathcal{O}_{\mathbb{F}_e}(-D))$. Taking higher direct
image we obtain on $Hilb^{l}(\mathbb{F}_e)$ the sheaf:
\[\Gamma := R^1_{p_{{2}_*}}\mathcal{H}om(p_{1}^*\mathcal{O}_{\mathbb{F}_e}(D - \beta F) \otimes
\mathcal{I}_{\mathcal{Z}_l},
p_{1}^*\mathcal{O}_{\mathbb{F}_e}(-D)).\]

From the semicontinuity Theorem \cite[Theorem 12.8]{Hartshorne1}, we have that the sets
\[\begin{aligned}
V_1 &:= \{Z \in Hilb^l(\mathbb{F}_e) : h^0(\mathbb{F}_e,  \mathcal{O}_{\mathbb{F}_e}(aC_e+(2a-2b-\beta-1)F \otimes I_Z)< 1\} \\
V_2 &:= \{Z \in Hilb^l(\mathbb{F}_e) :h^0(\mathbb{F}_e, \mathcal{O}_{\mathbb{F}_e}((2a-1)C_e)+(a-2b-\beta-1)F \otimes I_Z)< 1\}
\end{aligned}\]
are open sets of $Hilb^l(\mathbb{F}_e)$.  By Theorem \ref{TP1}, the set $V:=V_1 \cap V_2$ is non-empty.  Restricting the sheaf $\Gamma$ to $V$ we have that it is locally free because
\[\begin{aligned}
H^0(\mathcal{H}om(p_{1}^*\mathcal{O}_{\mathbb{F}_e}(D - \beta F) & \otimes
\mathcal{I}_{\mathcal{Z}_l},
p_{1}^*\mathcal{O}_{\mathbb{F}_e}(-D))) \cong \\
& Hom(\mathcal{O}_{\mathbb{F}_e}(D - \beta F) \otimes I_Z,
\mathcal{O}_{\mathbb{F}_e}(-D))=0,
\end{aligned}\] 
and
\[\begin{aligned}
\dim \,
Ext^2(\mathcal{O}_{\mathbb{F}_e}(D - \beta F) & \otimes I_Z,
\mathcal{O}_{\mathbb{F}_e}(-D)) = \\
& h^0(\mathbb{F}_e, \mathcal{O}_{\mathbb{P}^2}(2D -2 C_e - (e+\beta+2) F))\otimes I_Z)=0
\end{aligned}\] 
for any $Z
\in V$. Hence, the fiber over $Z \in V$ is
$Ext^1(\mathcal{O}_{\mathbb{F}_e}(D - \beta F) \otimes I_Z,
\mathcal{O}_{\mathbb{F}_e}(-D))$.

Let $\mathbb{P}\Gamma$ denote  the projectivization of the sheaf  $\Gamma$ on $V$.
By \cite[Lemma
3.2]{Gottsche} there exists an exact sequence:
\begin{equation} \label{univextension2}
0 \to (id\times\pi)^*p_{1}^*\mathcal{O}_{\mathbb{F}_e}(-D) \otimes
\mathcal{O}_{\mathbb{F}_e \times \mathbb{P}\Gamma}(1) \to
\mathcal{E} \to
(id\times\pi)^*(p_{1}^*\mathcal{O}_{\mathbb{F}_e}(D-\alpha C_e - \beta F) \otimes
\mathcal{I}_{\mathcal{Z}_l}) \to 0
\end{equation}
on $\mathbb{F}_e \times \mathbb{P}\Gamma$ such that for each $p
\in \mathbb{P}\Gamma$ the restriction $\mathcal{E}_{|_p}$ of
$\mathcal{E}$ to $\mathbb{F}_e \times \{p\}$ is isomorphic to an
extension
\[0 \longrightarrow \mathcal{O}_{\mathbb{F}_e}(-D)  \longrightarrow E  \longrightarrow
\mathcal{O}_{\mathbb{F}_e}(D - \beta F) \otimes I_Z \longrightarrow 0.\]
Define the set
\[U:= \{p \in  \mathbb{P}\Gamma  : \text{$\mathcal{E}_{|_p}$ is stable and $S_H(\mathcal{E}_p)=s$}\}.\]
From Theorem \ref{TP1}, the lower semicontinuity of the
function $S_H$ and the fact that stability is an open condition we
conclude that the set $U$ is non-empty and open in $\mathbb{P}
\Gamma$. Restricting the sequence (\ref{univextension2}) to
$\mathbb{F}_e \times U$  we have, from the universal property of
the moduli space $M_{\mathbb{F}_e}(2; -\beta F,c_2)$, a morphism
\[f_s:U \longrightarrow M_{\mathbb{F}_e}(2; -\beta F,c_2)\]
where $f_s(U)$ is contained in the stratum
$M_{\mathbb{P}^2}(2; -\beta F,c_2;s)$. Hence,
$f_s(U)$,  being the image of an irreducible
variety under a morphism is irreducible.

We can now determine a lower bound for the dimension of the stratum;  
\[
 \dim\, M_{\mathbb{F}_e}(2;-\beta F,c_2;s)  \geq dim \, Im (f_s(U)) = \dim \, U - \dim f_s^{-1}(E) \]
 which is equivalent to:
 \[\begin{aligned}
    \dim\, &  M_{\mathbb{F}_e}(2;-\beta F,c_2;s)  \geq  \\ 
    & \dim \,V + \dim \, Ext^1(
\mathcal{O}_{\mathbb{F}_e}(D - \beta F) \otimes I_Z,
\mathcal{O}_{\mathbb{F}_e}(-D)) - \dim \,\mathbb{P}H^0(E(D))-1. 
 \end{aligned}
\]

Since $V$ is an open set of $Hilb^l(\mathbb{F}_e)$, we have
\[\begin{aligned}
    \dim\, & M_{\mathbb{F}_e} (2;-\beta F,c_2;s)  \geq  \\ & \dim \,Hilb^l(\mathbb{F}_e) + \dim \, Ext^1(
\mathcal{O}_{\mathbb{F}_e}(D - \beta F) \otimes I_Z,
\mathcal{O}_{\mathbb{F}_e}(-D)) - \dim \,\mathbb{P}H^0(E(D))-1,
\end{aligned} 
\]

We now compute the values of \[\dim \, Ext^1(
\mathcal{O}_{\mathbb{F}_e}(D - \beta F) \otimes I_Z,
\mathcal{O}_{\mathbb{F}_e}(-D) )
\text{ and } \dim \,\mathbb{P}H^0(E(D)),\] where $Z \in V$.

Note that by Serre duality
$Ext^1(
\mathcal{O}_{\mathbb{F}_e}(D - \beta F) \otimes I_Z,
\mathcal{O}_{\mathbb{F}_e}(-D))$ is canonically dual to \\
$Ext^1(\mathcal{O}_{\mathbb{F}_e}(-D),
\mathcal{O}_{\mathbb{F}_e}(D-2 C_e - (\beta+e+2) F) \otimes I_Z)$.  Since
$\mathcal{O}_{\mathbb{F}_e}(-D)$ is locally free, then
\[Ext^1(\mathcal{O}_{\mathbb{F}_e}(-D),
\mathcal{O}_{\mathbb{F}_e}(D-2 C_e - (\beta+e+2) F) \otimes I_Z) \cong
H^1(\mathbb{F}_e, \mathcal{O}_{\mathbb{P}^2}(2D-2C_e-(e+2+\beta)F) \otimes I_Z ).\]

By the exact sequence
\[0 \longrightarrow \mathcal{O}_{\mathbb{P}^2}(2D-2C_e-(e+2+\beta)F) \otimes I_Z  \longrightarrow \mathcal{O}_{\mathbb{P}^2}(2D-2C_e-(e+2+\beta)F)
\longrightarrow \mathcal{O}_Z \longrightarrow 0\]
we have that
\begin{equation} \label{dimension}
    h^1(\mathcal{O}_{\mathbb{F}_e}(2D-2C_e-(e+2+\beta)F ) \otimes I_Z)= h^1(\mathcal{O}_{\mathbb{F}_e}(2D-2C_e-(e+2+\beta)F) + h^0(\mathcal{O}_Z).
\end{equation}
  
Since $h^2(\mathbb{F}_e,\mathcal{O}_{\mathbb{F}_e}(2D-2C_e-(e+2+\beta)F) = h^0(\mathbb{F}_e,  -2D+\beta F)=0$, from Theorem \ref{CH}  we have
\[h^1(\mathcal{O}_{\mathbb{F}_e}(2D-2C_e-(e+2+\beta)F )  = \frac{(2a-1)}{2}(4b+2ae-2\beta-2).\]

Therefore:
 \begin{equation}\label{Ext1}
  h^1(\mathcal{O}_{\mathbb{F}_e}  (2D-2C_e-(e+2+\beta)F ) \otimes I_Z) = l(Z)+ \frac{(2a-1)(4b+2ae-2\beta-2)}{2}.
 \end{equation}

Now, we compute $h^0(\mathbb{F}_e,E(D))$. Since $E \in
M_{\mathbb{F}_e}(2; - \beta F,c_2;s)$ it can be written in an extension
\begin{equation}
    0 \longrightarrow \mathcal{O}_{\mathbb{F}_e}(-D) \longrightarrow E \longrightarrow \mathcal{O}_{\mathbb{F}_e}(D -\beta F)\otimes I_Z \longrightarrow 0
\end{equation}
from which we get:
\[0 \longrightarrow \mathcal{O}_{\mathbb{F}_e} \longrightarrow E(D) \longrightarrow \mathcal{O}_{\mathbb{F}_e}(2D -\beta F)\otimes I_Z \longrightarrow 0.\]
Since the divisor $2D -\beta F$ is not effective on $\mathbb{F}_e$, we have
\begin{equation}
    \label{fiber}
h^0(\mathbb{F}_e,E(D))  = 1 + h^0(\mathcal{O}_{\mathbb{F}_e}(2D-\alpha C_e -\beta F)\otimes I_Z) =1.
\end{equation}

Replacing (\ref{Ext1}) and (\ref{fiber}) in (\ref{dimension}), we
have
\[
 \dim\, M_{\mathbb{F}_e}(2;-\alpha C_e -\beta F,c_2;s) \geq  3l+ \frac{(2a-1)(4b+2ae+2\beta+2)}{2}.
\]
Therefore, \[
 \dim\, M_{\mathbb{F}_e}(2; -\beta F,c_2;s) \geq  \max_{a,b} \left\{3l+  \frac{(2a-1)(4b+2ae+2\beta+2)}{2} \right\}
\]
where the maximum is taken over all pairs $(a,b)$ satisfying the hypothesis of the proposition and the proof is complete.
\end{proof}

\begin{Proposition} \label{dimensionstratum2}
\begin{enumerate}
    \item  Let $H=C_e+(e+1)F$  be an ample divisor on $\mathbb{F}_e$, $e \geq 0$ and let $\alpha, \beta \in \{0,1\}$, $c_2 \geq 2$, $s, l \geq 0$.  Suppose that there exist $a,b \in \mathbb{N}$, $b \neq 0$ such that $s= 2(b-a)-\alpha-\beta$ and $l= c_2-a^2e-ae\alpha-2ab+a\beta-b\alpha >0$.  Let $D:= -aC_e+bF$ be a divisor on $\mathbb{F}_e$ and let $Z \subset \mathbb{F}_e$ be a local complete intersection of codimension two and lenght $l$
satisfiying the hypothesis of  Theorem \ref{TP2}.
   Then,  the stratum $M_{\mathbb{F}_e,H}(2;-\alpha C_e-\beta F,c_2;s)$ is non-empty and it has an irredubible subvariety of dimension 
   \[
   \max_{a,b}\left\{3l+  \frac{(2a+\alpha+1)(4b-2ae-\alpha e-2\beta-4e-2)}{2}\right\}.
\]
where the maximum is taken over all pairs $(a,b)$ satisfying the hypothesis of the proposition. Moreover, 
\[
 \dim\, M_{\mathbb{F}_e}(2;-\alpha C_e -\beta F,c_2;s) \geq \max_{a,b}\left\{3l+  \frac{(2a+\alpha+1)(4b-2ae-\alpha e-2\beta-4e-2)}{2}\right\}\]
 
 \item Let $H=C_e+(e+1)F$  be an ample divisor on $\mathbb{F}_e$, $0 \leq e$ and let $\alpha, \beta \in \{0,1\}$, $c_2 \geq 2$, $s,l > 0$.  Suppose that there exist $a,b \in \mathbb{N}-\{0\}$, such that $s= 2(a+b)-\alpha-\beta$ and $l= c_2-a^2e+ae\alpha+2ab-a\beta-b\alpha > 0$. Let $D:= aC_e+bF$ be a divisor on $\mathbb{F}_e$ and let $Z \subset \mathbb{F}_e$, be a local complete intersection of codimension two and lenght  $l$
    satisfiying the hypothesis of  Theorem \ref{TP3}.
   Then, the stratum $M_{\mathbb{F}_e,H}(2;-\alpha C_e-\beta F,c_2;s)$ is non-empty and it has an irredubible subvariety of dimension
   \[ \max_{a,b} \left\{ 3l+  \frac{(2a-\alpha-1)(2\beta+2ae-\alpha -4b+2)}{2}\right\}.\]
 where the maximum is taken over all pairs $(a,b)$ satisfying the hypothesis of the proposition.  Moreover, 
\[
 \dim\, M_{\mathbb{F}_e}(2;-\alpha C_e -\beta F,c_2;s) \max_{a,b} \geq \left\{ 3l+  \frac{(2a-\alpha-1)(2\beta+2ae-\alpha -4b+2)}{2}\right\}.\]
 \end{enumerate}
\end{Proposition}

\begin{proof}
The proof follows by the same method as in the proof of Proposition \ref{dimensionstratum}. Noting that for item (3), the property of being locally free is an open condition (see \cite[Lemma 2.1.8]{Huybrechts-Lehn}).
\end{proof}

Now we are able to prove Theorems \ref{dimmax1} and \ref{dimmax2}.

\begin{proof}
\textbf{Proof of Theorem \ref{dimmax1} and Theorem \ref{dimmax2}}. The proof follows directly of Proposition \ref{dimensionstratum}, and Proposition \ref{dimensionstratum2}  noting that the maximum is taken over all numbers $a_i,b_i$ that satisfies  the hypothesis of theorem.
\end{proof}

\begin{Corollary} 
Let $D$ be a divisor on $\mathbb{F}_e$, $e \geq 0$.
\begin{enumerate}
    \item Under the hypothesis of Theorem \ref{dimmax1}.  Then, the stratum $M_{\mathbb{F}_e,H}(2;-\beta F +D,c_2;s)$ has an irredubible subvariety of dimension
\[\begin{aligned}
     \max_{a_i,b_i} \{ 3l+ \frac{1}{2} (2a_1-1)(4b_1+2a_1e+2\beta+2),  & 3l+ \frac{1}{2}(2a_2+1)(4b_2-2a_2e-2\beta-4e-2), \\ & 3l+ \frac{1}{2}(2a_3-1)(2\beta+2a_3e-4b_3+2) \}
\end{aligned}\]
where the maximum is taken over all numbers $a_i,b_i$ that satisfies  the hypothesis of Theorem \ref{dimmax1}. Moreover, 
\[\begin{aligned}
   \dim\, & M_{\mathbb{F}_e}(2; -\beta F+D,c_2;s) \geq 
    \max_{a_i,b_i} \{ 3l+ \frac{1}{2} (2a_1-1)(4b_1+2a_1e+2\beta+2), \\ & 3l+ \frac{1}{2}(2a_2+1)(4b_2-2a_2e-2\beta-4e-2), 3l+ \frac{1}{2}(2a_3-1)(2\beta+2a_3e-4b_3+2) \}
\end{aligned}\] 

\item Under the hypothesis of Theorem \ref{dimmax2}.  Then, the stratum $M_{\mathbb{F}_e,H}(2;- C_e-\beta F +D,c_2;s)$ has an irredubible subvariety of dimension
\[\begin{aligned}
     \max_{a_i,b_i}\{
   3l+ \frac{1}{2} (2a_1-2)(4b_1+2a_1e- e+2\beta+2), &
    3l+ \frac{1}{2}(2a_2+2)(4b_2-2a_2e- e-2\beta-4e-2),\\  & 3l+ \frac{1}{2}(2a_3-2)(2\beta+2a_3e -4b_3+1) \}
   \end{aligned}
\]
where the maximum is taken over all numbers $a_i,b_i$ that satisfies  the hypothesis of Theorem \ref{dimmax2}. Moreover, 
\[\begin{aligned}
   \dim\, & M_{\mathbb{F}_e}(2;-\alpha C_e -\beta F+D,c_2;s) \geq 
    \max_{a_i,b_i}\{
   3l+ \frac{1}{2} (2a_1-2)(4b_1+2a_1e- e+2\beta+2), \\
   & 3l+ \frac{1}{2}(2a_2+2)(4b_2-2a_2e- e-2\beta-4e-2), 3l+ \frac{1}{2}(2a_3-2)(2\beta+2a_3e -4b_3+1) \}
   \end{aligned}\]
\end{enumerate}
\end{Corollary}

\begin{proof}
We only prove $(1)$, the proof of $(2)$ follows by the same method as in $(1)$. Since $S_H(E) = S_H(E \otimes \mathcal{O}_{\mathbb{F}_e}(D))$, it follows that the map
\[\begin{aligned}
    M_{\mathbb{F}_e,H}(2;-\alpha C_e-\beta F ,c_2;s) &\rightarrow M_{\mathbb{F}_e,H}(2;-\alpha C_e-\beta F +D,c_2;s) \\ 
    E &\mapsto E \otimes \mathcal{O}_{\mathbb{F}_e}(D)
\end{aligned}\]
is an isomorphism.
\end{proof}

\section{Applications  to Brill-Noether Theory and Change of Polarizations}

In this section, we use the previous results to give information about the following problems: 

Let $H_t:= C_e+(e+t)F$, $t> 0$ be an ample divisor on $\mathbb{F}_e$, $e \geq 0$. 
\begin{itemize}

   \item [(i)] Describe the differences between the moduli spaces $M_{\mathbb{F}_e, H_t}(2;-C_e-\beta F,c_2)$ where $\beta \in \{0,1\}$.  

    \item [(ii)] The non-emptiness of some  Brill-Noether loci in the moduli space
$M_{\mathbb{F}_e, H_1}(2;c_1,c_2)$ of stable vector bundles of  rank
$2$ and fixed Chern classes $c_1$ and $c_2$ on $\mathbb{F}_e$. 
\end{itemize}

\subsection{Differences between the moduli spaces $M_{\mathbb{F}_e, H_t}(2;c_1,c_2)$.}

Let $X$ be a smooth, irreducible projective surface. In \cite{Qin}, Qin considered the problem: What is the difference between $M_{X, H_1}(2;c_1,c_2)$ and $M_{X, H_2}(2;c_1,c_2)$ ? where $H_1$, $H_2$ are two different ample line bundles on $X$.

\begin{Definition} \emph{( \cite[Definition3.2]{Costa-Miro-Roig})
Let $C_X$  denote the K\"ahler cone in $Num(X) \otimes \mathbb{R}$ generated by all ample divisors.
\begin{itemize}
\item Let $\xi \in Num(X) \otimes \mathbb{R}$. We define
\[\mathcal{W}^\xi := C_X \cap \{x \in Num(X)\otimes \mathbb{R} : x\xi =0\}.\]
\item Define $\mathcal{W}(c_1,c_2)$ as the set whose elements consist of $\mathcal{W}^\xi$, where $\xi$ is the numerical equivalence class of a divisor $D$ on $X$ such that $\mathcal{O}_X(D+c_1)$ is divisible by $2$ in $Pic(X)$, and that 
\[D^2 < 0, \,\,\,\, c_2+ \frac{D^2-c_1^2}{4}= l(Z)\]
for some locally complete intersection codimension $2$ in $X$.
\item A wall of type $(c_1,c_2)$ is an element in $\mathcal{W}(c_1,c_2)$.  A chamber of type $(c_1,c_2)$ is a connected component of $C_X-\mathcal{W}(c_1,c_2)$.  A $\mathbb{Z}-$chamber of type $(c_1,c_2)$ is the intersection of $Num(X)$ with some chamber of type $(c_1,c_2)$.
\item A face of type $(c_2,c_2)$ is $F = \mathcal{W}^\xi$ is a wall of type $(c_1,c_2)$ and $C$ is a chamber of type $(c_1,c_2)$.
\end{itemize}
We say that a wall $\mathcal{W}^\xi$ of type $(c_1,c_2)$ separates two polarizations $H$ and $H'$, if and only if $\xi \cdot H < 0 < \xi \cdot H'$.}
\end{Definition}

\begin{Definition} \emph{( \cite[Definition 3.4]{Costa-Miro-Roig})
Let $\xi$ be a numerical equivalence class defining a wall of type $(c_1,c_2)$.  Denote by $\xi_{\xi}(c_1,c_2)$  the quasiprojective variety parametrizing rank $2$ vector bundles $E$ on $X$ given by an extension
\[0 \rightarrow \mathcal{O}_X(D) \rightarrow E \rightarrow \mathcal{O}_X(c_1-D) \otimes I_Z \rightarrow 0\]
where  $D$, is a divisor with $2D-c_1 \equiv \xi$ and $Z$ is a locally complete intersection of codimension $2$ of lenght $c_2+ \frac{\xi^2-c_1^2}{4}$.}
\end{Definition}

\begin{Remark} \label{Qin}\emph{
Qin (\cite{Qin}), proved that the moduli space  $M_{X, H}(2;c_1,c_2)$ only depends on the chamber of $H$ and the study of moduli spaces of rank $2$ stable vector bundles with respect to a polarization lying on walls may be reduced to the study of moduli spaces of rank two stable vector bundles with respect to a polarization lying on $\mathbb{Z}-$chambers. Moreover, let $H_1$, $H_2$ be two ample line bundles on $X$ lying on chambers $C_1$ and $C_2$
, sharing a common wall, we have
\begin{equation} \label{change}
    M_{H_1}(2;c_1,c_2) = M_{H_2}(2;c_1,c_2)- ( (\cup_\xi \xi_\xi(c_1,c_2)) \cup (\cup_{\xi} \xi_{-_{\xi}}(c_1,c_2))
\end{equation}
where $\xi$ satisfies $\xi L_1 > 0$, and runs over all numerical equivalences classes which define the common wall $\mathcal{W}$.}
\end{Remark}

The folllowing propositions give information about of the changes between the moduli spaces  $M_{\mathbb{F}_e, H_1}(2;-C_e-\beta F,c_2)$, and $M_{\mathbb{F}_e, H_t}(2;-C_e-\beta F,c_2)$,  where  $H_t:= C_e+(e+t)F$, $t> 0$ is an ample divisor on $\mathbb{F}_e$, $e \geq 0$.

\begin{Proposition} \label{PropChanges}
Let $\beta \in \{0,1\}$, $H_t:= C_e+(e+t)F$, $t> 0$ be an ample divisor on $\mathbb{F}_e$, $e \geq 0$,  and $s:= 2N-\beta-1 >0$, $N \in \mathbb{N}-\{0\}$.
Let $E \in M_{\mathbb{F}_e, H_1}(2;-C_e -\beta F,c_2;s)$.  Then
\begin{itemize}
    \item [(i)] If  $\mathcal{O}_{\mathbb{F}_e}(-aC_e+bF) \subset E$  is maximal and $2tN \leq 2b(1-t)+t+\beta$, then $E \notin M_{\mathbb{F}_e, H_t}(2;-C_e -\beta F,c_2)$.
    
    \item [(ii)] If  $\mathcal{O}_{\mathbb{F}_e}(aC_e-bF) \subset E$  is maximal and $2N \leq 2a(t+1)+t+\beta$, then $E \notin M_{\mathbb{F}_e, H_t}(2;-C_e -\beta F,c_2)$.  
    
    \item [(iii)] If  $\mathcal{O}_{\mathbb{F}_e}(-aC_e-bF) \subset E$  is maximal and $2tN \leq 2b(t-1)+t+\beta$, then $E \notin M_{\mathbb{F}_e, H_t}(2;-C_e -\beta F,c_2)$.
\end{itemize}
Moreover, any $L$ satifying (i), (ii) or (iii) defines a non-empty wall  $\xi:=2L+C_e+ \beta  F$ separating $H_1$ and $H_t$ and $E \in \xi_{\xi}(-C_e-\beta F,c_2)$. 
\end{Proposition}

\begin{proof}
 Let $E \in M_{\mathbb{F}_e, H_1}(2;-C_e-\beta F,c_2;s)$, by Remark \ref{NotUnique}, there exists a line bundle $L \subset E$ such that \[s:= -2(c_1(L)\cdot H_1)-\beta-1 > 0. \]
 Let $-N$ denote the value  $c_1(L) \cdot H_1$. Since $E$ is $H_1-$stable and $Pic(\mathbb{F}_e) \cong \mathbb{Z}C_e \oplus \mathbb{Z}F$, we have that 
 
 \begin{itemize}
     \item [(i)] If  $L=\mathcal{O}_{\mathbb{F}_e}(-aC_e+bF)$, $a \neq 0$, then \[\frac{-t-\beta}{2}=\mu_{H_t}(E) \leq c_1(L) \cdot H_t = b(1-t)-tN\]
whenever $2tN \leq 2b(1-t+t+\beta)$. Therefore $E \notin M_{\mathbb{F}_e, H_t}(2;-C_e -\beta F,c_2)$.

     \item [(ii)] If  $L=\mathcal{O}_{\mathbb{F}_e}(aC_e-bF)$, $b\neq 0$, then \[\frac{-t-\beta}{2}=\mu_{H_t}(E) \leq c_1(L) \cdot H_t = a(t+1)-N\]
whenever $2N \geq 2a(t+1)+t+\beta$. Therefore $E \notin M_{\mathbb{F}_e, H_t}(2;-C_e -\beta F,c_2)$.

     \item [(iii)] If  $L=\mathcal{O}_{\mathbb{F}_e}(-aC_e-bF)$, $a,b \neq 0$, then \[\frac{-t-\beta}{2}=\mu_{H_t}(E) \leq c_1(L) \cdot H_t = b(t-1)-tN\]
whenever $2tN \leq 2b(t-1)+t+\beta$. Therefore $E \notin M_{\mathbb{F}_e, H_t}(2;-C_e -\beta F,c_2)$.
 \end{itemize}
Hence, if $L$ satisfies (i), (ii) or (iii), then $E$ is $H_1-$stable but $H_t-$unstable and $E \in \xi_\xi(-C_e-\beta F)$ where $\xi:=2L+C_e+ \beta  F$ defines a non-empty wall of type $(-C_e -\beta F,c_2)$ separating $H_1$ and $H_t$.
\end{proof}

\begin{Proposition}
     Let $\beta \in \{0,1\}$, $H_t:= C_e+(e+t)F$, $t> 0$ be an ample divisor on $\mathbb{F}_e$, $e \geq 0$,  and $s:= 2N-\beta-1 >0$, $N \in \mathbb{N}-\{0\}$.  Let $\xi:= 2L+C_e+\beta F$ satisfying any of the conditions of Proposition \ref{PropChanges}.  Then
     \[
    M_{H_1}(2;c_1,c_2) = M_{H_2}(2;c_1,c_2)- ( (\cup_\xi \xi_\xi(c_1,c_2)) \cup (\cup_{\xi} \xi_{-_{\xi}}(c_1,c_2))
\]
\end{Proposition}

\begin{proof}
The proof follows directly from Proposition \ref{PropChanges} and Remark \ref{Qin}.
\end{proof}

\subsection{Brill-Noether theory}

Let $H$ be an ample divisor on $\mathbb{F}_e$, $e \geq 0$.  For any $k \geq 0$, the subvariety of $M_{\mathbb{F}_e,H}(2;c_1,c_2)$ 
defined as
\[W^k(2;c_1,c_2) := \{E \in M_{\mathbb{F}_e,H}(2;c_1,c_2) : h^0(E) \geq k\},\]
It is called the \textit{$k-$Brill-Noether locus} of the moduli space
$M_{\mathbb{F}_e,H}(2;c_1,c_2)$ (or simply Brill-Noether locus, if
there is no confusion) (see 
\cite{Costa-Miro-Roig} for more details).

The following theorem yields information about the variety
$W^k(2;c_1,c_2)$, in particular shows that $W^k(2;c_1,c_2)$ is a
determinantal variety and, gives a formula for the expected
dimension. 

\begin{Theorem}\cite[Corollary 2.8]{Costa-Miro-Roig} \label{determinantalvariety}
 Let $H$ be an ample divisor on $\mathbb{F}_e$, $e \geq 0$. Let $M_{\mathbb{F}_e,H}(2;c_1,c_2)$ be the moduli space of stable vector bundles of rank $2$ on $\mathbb{F}_e$,  with fixed Chern classes $c_1$, $c_2$.  Then, for any $k \geq 0$, there exists a determinantal variety
 \[W^k(2;c_1,c_2):= \{E \in M_{\mathbb{F}_e,H}(2;c_1,c_2) : h^0(E) \geq k\}.\]
 Moreover, each non-empty irreducible component of $W^k(2;c_1,c_2)$ has dimension greater or equal to the Brill-Noether number on $\mathbb{F}_e$
 \[\rho^k(2;c_1,c_2) := 4c_2-c_1^2-3 - k\left(k-\frac{c_1(c_1-K_{\mathbb{F}_e})}{2}+c_2-2 \right)\]
 and
 \[W^{k+1}(2;c_1,c_2) \subset Sing(W^k(2;c_1,c_2))\]
 whenever $W^k(2;c_1,c_2) \neq M_{\mathbb{F}_e,H}(2;c_1,c_2)$.
\end{Theorem}

The following result allows to  establish a relationship between
the Brill-Noether locus $W^k(2;c_1,c_2)$ and the different strata
$M_{\mathbb{F}_e,H}(2;c_1,c_2;s)$:

\begin{Proposition}
\begin{enumerate}
    \item Under the hypotheses of Proposition \ref{dimensionstratum}.  Let $(a,b) \in \mathbb{N}$, $a\neq 0$ be the pair for which the number 
      \begin{equation*} 
          \left\{3l+  \frac{(2a-\alpha-1)(4b+2ae-\alpha e+2\beta+2)}{2} \right \}
      \end{equation*} 
      is maximum. Let $D := aC_e$ be a divisor on $\mathbb{F}_e$.  Then, the Brill-Noether locus $W^{b+1}(2;(2a-\alpha)C_e-\beta F, c_2)$, is non-empty, and 
      \[dim\, W^{b+1}(2;(2a-\alpha)C_e-\beta F, c_2) \geq \left\{3l+  \frac{(2a-\alpha-1)(4b+2ae-\alpha e+2\beta+2)}{2} \right \}.\] 
      Moreover, $\rho^{b+1}(2, (a-\alpha)C_e-\beta F, c_2) < \left\{3l+  \frac{(2a-\alpha-1)(4b+2ae-\alpha e+2\beta+2)}{2} \right \}$ for $c_2 >> 0$.
      
      \item Under the hypotheses of item (1) Proposition \ref{dimensionstratum2}.  Let $(a,b) \in \mathbb{N}$, $b\neq 0$ be the pair for which the number 
      \begin{equation} \label{App1}
          \left\{3l+  \frac{(2a+\alpha+1)(4b-2ae-\alpha e-2\beta-4e-2)}{2} \right \}
      \end{equation} 
      is maximum. Let $D := bF$ be a divisor on $\mathbb{F}_e$.  If $ae \leq 1$,  then the Brill-Noether locus $W^{k}(2;-\alpha C_e+(2b-\beta) F, c_2)$, $k=(a+1)-\frac{ea(a+1)}{2}$ is non-empty, and 
      \[dim \, W^{k}(2;-\alpha C_e+(2b-\beta) F, c_2) \geq \left\{3l+  \frac{(2a+\alpha+1)(4b-2ae-\alpha e-2\beta-4e-2)}{2} \right \}.\]
      Moreover, $\rho^{k}(2, -\alpha C_e+(2b-\beta) F, c_2) < \left\{3l+  \frac{(2a+\alpha+1)(4b-2ae-\alpha e-2\beta-4e-2)}{2} \right \}$ for $c_2 >> 0$.
      
      \item  Under the hypotheses of item (2) Proposition \ref{dimensionstratum2}.  Let $(a,b) \in \mathbb{N}$, $a\neq 0$ be the pair for which the number 
      \begin{equation} 
          \left\{3l+  \frac{(2a-\alpha-1)(2\beta+2ae-\alpha -4b+2)}{2} \right \}
      \end{equation} 
      is maximum. Let $D := aC_e+bF$ be a divisor on $\mathbb{F}_e$.  Then, the Brill-Noether locus $W^{1}(2;(2a-\alpha)C_e-(2b-\beta) F, c_2)$, is non-empty, and 
      \[\dim \, W^{1}(2;(2a-\alpha)C_e+(2b-\beta) F, c_2) \geq \left\{3l+  \frac{(2a-\alpha-1)(2\beta+2ae-\alpha -4b+2)}{2} \right \}.\]
      Moreover, $\rho^{1}(2, (2a-\alpha)C_e+(2b-\beta) F, c_2) < \left\{3l+  \frac{(2a-\alpha-1)(2\beta+2ae-\alpha -4b+2)}{2} \right \}$ for $c_2 >> 0$.
\end{enumerate}
\end{Proposition}

\begin{proof}
 The proof follows directly of Proposition \ref{dimensionstratum} and Proposition \ref{dimensionstratum2}.
\end{proof}


\begin{thebibliography}{99}

\bibitem{Aprodu}
{\sc M.\ Aprodu,  V. \ Brînzǎnescu, and M.\ Marchitan.}---
    {\it Rank-two vector bundles on Hirzebruch surfaces}. Cent. Eur. J. Math. 10 (2012), no. 4, 1321–1330. 


\bibitem{Aprodu-Brinzanescu}
{\sc M.\ Aprodu, and  V. \ Brînzǎnescu.}---
    {\it Moduli spaces of vector bundles over ruled surfaces}. Nagoya Math. J. 154 (1999), 111–122.

\bibitem{Costa-Miro-Roig}
 {\sc L.\ Costa, and R.M. \ Miro-Roig.}---
    {\it Brill-Noether theory for moduli spaces of sheaves on algebraic varieties}. Forum Math. 22 (2010), no. 3, 411-432.

\bibitem{Coskun-Huizenga}
{ \sc I. \  Coskun, and J. \ Huizenga. }--- Weak {\it Brill-Noether theorems and globally generated vector bundles on Hirzebruch surfaces.},  Nagoya Math. J. 238 (2020), 1–36. 14J60 (14J26)

\bibitem{Coskun}
{ \sc I. \  Coskun, and J. \ Huizenga. }--- Weak {\it Brill-Noether
for rational surfaces. Local and global methods in algebraic
geometry},  Contemp. Math., 712, Amer. Math. Soc., Providence, RI,
2018.


\bibitem {Friedman}
  {\sc R.\ Friedman.} ---
  {\it Algebraic surfaces and holomorphic vector bundles}. Universitext. Springer-Verlag, New York, 1998. x+328 pp.


\bibitem{Gottsche0}
   {\sc L. \ Gottsche, and A. \  Hirschowitz}.--- {\it  Weak Brill-Noether for vector bundles on the projective plane}.
   Algebraic geometry (Catania, 1993/Barcelona, 1994), Lecture Notes in Pure and Appl. Math., 200, Dekker, New York, 1998.


\bibitem {Gottsche}
  {\sc L.\ Gottsche.} ---
  {\it Change of polarization and Hodge numbers of moduli spaces of torsion free sheaves on surfaces.}
 Mathematische Zeitschrift, 223, 247-260, (1996).


\bibitem{Lange} {\sc H. \ Lange}.--- {\it Zur Klassifikation von
Regelmannigfaltigkeiten.}  Mathematische Annalen. 262 (4),
447-459, (1983).


\bibitem {Lange-Narasimhan}
  {\sc H.\ Lange, and M.S. \ Narasimhan.} ---
  {\it Maximal subbundles of Rank two vector bundles on curves.}
 Mathematishe Annalen, 266, 55-72, (1983).

 \bibitem {LePotier}
  {\sc J.\ Le Potier. } ---
  {\it Lectures on vector bundles}. Translated by A. Maciocia. Cambridge Studies in Advanced Mathematics, 54. Cambridge University Press, Cambridge, 1997. viii+251 pp.


   \bibitem {Hartshorne1}
  {\sc R.\ Hartshorne. } ---
   {\it Algebraic geometry}. Graduate Texts in Mathematics, No. 52. Springer-Verlag, New York-Heidelberg, 1977. xvi+496 pp.

\bibitem {Huybrechts-Lehn}
  {\sc D.\ Huybrechts, and  M. \ Lehn.} ---
   {\it The geometry of moduli spaces of sheaves.} Aspects of Mathematics, E31. Friedr. Vieweg  Sohn, Braunschweig, 1997. xiv+269 pp.

\bibitem{Maruyama}
  {\sc  M.\ Maruyama.} ---
{\it Openness of a family of torsion free sheaves}. J. Math. Kyoto
Univ. 16 (1976), no. 3, 627--637.

\bibitem{Maruyama1}
  {\sc M.\ Maruyama.} ---
  {\it Stable vector bundles on an algebraic surface.}
 Nagoya Math. J. Vol. 58 (1975), 25-68.
 
 \bibitem{Maruyama2}
 {\sc M.\ Maruyama.} ---
 {\it
 On classification of ruled surfaces. } Lectures in Mathematics, Department of Mathematics, Kyoto University, 3 Kinokuniya Book-Store Co., Ltd., Tokyo 1970 iv+75 pp.

\bibitem {Okonek}
  {\sc  C. \, Okonek, M. \, Schneider, and H. \, Spindler.} ---
  {\it Vector bundles on complex projective spaces}. Progress in Mathematics, 3. Birkhäuser, Boston, Mass., 1980. vii+389 pp.
  
\bibitem {Qin}
  {\sc  Qin, \, Z.} ---
  {\it Equivalence classes of polarizations and moduli spaces of sheaves}. J. Diff. Geom. 37 (1993), 397-415.


\bibitem{ALH}
{\sc L. \, Roa-Leguizam\'on, H. Torres-L\'opez, and A. G. \, Zamora} {\it On the Segre invariant for rank two vector bundles on $\mathbb{P}^2$}
Preprint. arXiv:2003.02727 
 
\end{thebibliography}
\end{document}